\documentclass[12pt,reqno]{amsart}%
\usepackage{amsmath, amsfonts, amssymb, amsthm, amscd, amsbsy,mathtools }
\usepackage[usenames,dvipsnames,svgnames,x11names,hyperref]{xcolor}
\usepackage{geometry}
\usepackage{graphicx,extarrows}
\usepackage{hyperref,url}
\usepackage{amsmath}
\usepackage{amsfonts}
\usepackage{amssymb}%

\usepackage{tabularx}
\usepackage{booktabs}
\usepackage{array}
\usepackage{ragged2e}

\newcolumntype{Y}{>{\RaggedRight\arraybackslash}X}
\providecommand{\U}[1]{\protect\rule{.1in}{.1in}}

\usepackage{geometry}
\usepackage{hyperref}
\hypersetup{
	colorlinks=true,
	allcolors=blue, 
}

\renewcommand{\epsilon}{\varepsilon}
\renewcommand{\phi}{\varphi}

\selectfont

\theoremstyle{plain}
\newtheorem{theorem}{Theorem}[section]

\theoremstyle{definition}

\newtheorem{remark}{Remark}

\theoremstyle{plain}
\newtheorem{lemma}[theorem]{Lemma}

\newtheorem{prop}[theorem]{Proposition}

\newtheorem{corollary}[theorem]{Corollary}

\newtheoremstyle{named}{}{}{\itshape}{}{\bfseries}{.}{.5em}{#1 \thmnote{#3}}
\theoremstyle{named}
\newtheorem*{namedthm*}{Theorem}

\newcommand{\R}{\mathbb{R}}
\newcommand{\dd}{\,\mathrm{d}}
\newcommand{\ee}{\mathrm{e}}

\begin{document}
\title{$L^p$-ASYMPTOTIC PROFILES FOR THE HEAT EQUATION WITH
HARDY POTENTIAL}

\author{Radu Ordean}
\address{
Department of Mathematics and Computer Science,
University of Bucharest,
Bucharest, Romania}
\email{raduordean@gmail.com}
\begin{abstract} For radial initial data, we construct explicit higher-order \(L^p(\mathbb R^N)\)-asymptotic profiles for the heat equation with Hardy potential. These profiles, denoted $A_n$ are obtained from the small-argument expansion, up to an arbitrary order \(n\), of the modified Bessel function appearing in the radial Hardy heat kernel. If $u$ is the mild solution generated by this kernel, we prove that the corresponding remainder $u(x,t)-A_n(x,t)$ admits a polynomial decay depending on $n$ in \(L^p(\mathbb R^N)\) as \(t\to\infty\). We also treat the non-radial case through spherical harmonics: each angular mode evolves according to a radial Hardy heat equation with a modified parameter, leading to finite and infinite angular expansion versions of the asymptotic profile under suitable summability assumptions. \end{abstract}

\maketitle
\section{Introduction}

The heat equation with inverse-square potential
\begin{equation}\label{problem}
    \begin{cases}
    u_t(x,t)=\Delta u(x,t)+\dfrac{\lambda}{|x|^2}u(x,t),\quad (x,t)\in \R^N\times(0,\infty)\\
    u(x,0)=u_0(x),\quad x\in \R^N
    \end{cases}
\end{equation}

is naturally associated with the Hardy inequality. It has been extensively studied since the work of Baras and Goldstein
\cite{BG}, who showed the phenomenon of instantaneous blow-up:
if \(\lambda>\lambda_*:=\left(\frac{N-2}{2}\right)^2\), then the problem
admits no positive local-in-time solution. Moreover, in the subcritical
and critical range \(0<\lambda\le \lambda_*\), they obtained necessary and
sufficient conditions on the initial data for the existence of a
nonnegative solution. The number $
    \lambda_*=\left(\frac{N-2}{2}\right)^2 $
is precisely the optimal constant in the Hardy inequality \cite{BM}:
\[
    \int_{\mathbb R^N} |\nabla u(x)|^2\,dx
    \ge
    \lambda_*
    \int_{\mathbb R^N}\frac{|u(x)|^2}{|x|^2}\,dx,
    \qquad u\in C_c^\infty(\mathbb R^N),
\]
and therefore plays a fundamental role in the qualitative behaviour of the equation. 

\subsection*{Motivation}

The study of large-time asymptotic profiles is a classical way of understanding the
dominant behaviour of parabolic equations. For the standard heat equation
\begin{equation}
    \begin{cases}
           u_t(x,t)=\Delta u(x,t) \qquad & (x,t)\in\mathbb R^N\times (0,\infty),\\
           u(x,0)=u_0(x)\quad &x\in \R^N,
    \end{cases}
\end{equation}
solutions with $u_0\in L^1(\R^N)$ initial data are asymptotically described by the Gaussian
heat kernel (see \cite{ZuazuaNotes}, Theorem 1.1). More precisely, if $
    M=\int_{\mathbb R^N}u_0(x)\,dx, $
then
\[
    \lim_{t\to \infty} t^{\frac N2(1-\frac1p)}
    \left\|
        u(\cdot,t)
        -
        M G(\cdot,t)
    \right\|_{L^p(\mathbb R^N)}
    =0,\qquad 1\leq p\leq \infty.
\]
Thus, the large-time behaviour is
determined only by the total mass of the initial datum.

For the problem \eqref{problem}, the inverse-square potential changes this picture. The singularity at the origin suggests that the total mass should be replaced
by a weighted moment adapted to the Hardy operator. 

Large-time behaviour for heat equations with inverse-square potentials has been
studied in several directions, beginning with the foundational work of Baras and
Goldstein~\cite{BG} on existence and instantaneous blow-up, and with the work of
Vázquez and Zuazua~\cite{VZ} on the role of Hardy inequalities in the asymptotic
analysis. More recently, Cazacu, Ignat and Manea~\cite{CIM} studied the
large-time behaviour of the heat equation with Hardy inverse-square potential on
corner spaces \(\mathbb R^{N-k}\times(0,\infty)^k\), obtaining optimal
polynomial decay rates and the first asymptotic profile in \(L^2\). Heat kernel
estimates for Schroedinger operators with Hardy-type potentials were obtained,
among others, by Moschini and Tesei~\cite{MT05}, Milman and Semenov~\cite{MS,MS04},
Filippas, Moschini and Tertikas~\cite{FMT}, and Ishige, Kabeya and
Ouhabaz~\cite{IKO}. Closer to the present problem, Pilarczyk~\cite{P} obtained
asymptotic results for the Hardy heat equation in weighted \(L^p(\mathbb R^N)\)
norms, for \(1\leq p\leq\infty\), using suitable estimates for the Hardy heat
kernel.

In this paper, we study the \(L^p(\mathbb R^N)\)-asymptotic behaviour of
solutions to the heat equation with inverse-square potential. In the radial
case, we construct explicit higher-order asymptotic profiles depending on
weighted moments of the initial datum. More precisely, for each
positive integer $n$, we construct a profile \(A_n\) such that, for admissible
values of \(p\),
\[
\lim_{t\to\infty}
t^{\gamma_p+n}
\left\|
u(\cdot,t)-A_n(\cdot,t)
\right\|_{L^p(\mathbb R^N)}
=0.
\]
where \(\gamma_p\) is defined below. Thus, each additional term in the
asymptotic expansion improves the decay rate of the remainder by one power of
\(t^{-1}\).

We also discuss the sharpness of this decay rate. Finally, by decomposing non-radial data into spherical harmonics, we extend the construction to finite angular expansions and, under an additional summability condition on the angular coefficients, to infinite spherical harmonic expansions.

\textbf{Notations.}
Throughout this article, we set
\[N\in \mathbb{N},\qquad N\geq 3,\qquad 0\leq\lambda\leq\frac{(N-2)^2}{4}:=\lambda_*\]
\[\mu=\frac{N-2}{2},\qquad \nu=\sqrt{\mu^2-\lambda},\qquad 1\leq p< \infty\]

\begin{equation}\label{gammapdef}
\gamma_p
:=
\frac N2\left(1-\frac1p\right)
-\frac{\mu-\nu}{2}.
\end{equation}
\begin{equation}\label{gammapldef}
    \gamma_{p,\ell}
=
\frac N2\left(1-\frac1p\right)
-
\frac{\mu-\nu_\ell}{2},\qquad \nu_{\ell}=\sqrt{(\mu+\ell)^2-\lambda}.
\end{equation}
For \(j\in\mathbb N\), we define the weighted moments
\begin{equation}\label{momentsdef}
    M_{\nu,j}(u_0)
:=
\int_{\mathbb R^N}
|x|^{\nu-\mu+2j}u_0(x)\,dx.
\end{equation}
For \(\alpha\ge0\) and \(k\in\mathbb N_0\), define
\[
\|f\|_{L^1_{\alpha,k}}
:=
\int_{\mathbb R^N}
|x|^{\alpha-\mu}
\bigl(1+|x|^{2k}\bigr)
|f(|x|)|\,dx.
\]
and 
\[L^1_{\alpha,k}(\R^N):=\left\{f \ \ \text{ measurable},\ \  \|f\|_{L^1_{\alpha,k}}<\infty\right\}\]
For radial initial data \(u_0\), we denote by
\[
u(x,t):=S_\lambda(t)u_0(x)=\int_{\mathbb R^N}
K(x,y,t)u_0(y)\,dy,
\]
the mild solution of the Hardy heat equation \eqref{problem}, where the radial Hardy heat kernel is given by
\begin{equation}\label{radialkerneldef}
    K(x,y,t) = \frac1{2t}(|x||y|)^{-\mu}\exp\left(-\frac{|x|^2+|y|^2}{4t}\right)I_\nu\left(\frac{|x||y|}{2t}\right),
\end{equation}
where \(I_\nu\) denotes the modified Bessel function of the first kind (see, for example, \cite{W}). Although the radial heat kernel formula is known (for example \cite[Theorem 3.18.1]{CCFI}), a complete derivation of the radial and
non-radial kernel in a form adapted to the Cauchy problem \eqref{problem} is not easy to find in the literature. For this reason, we include in the Appendix section a self-contained derivation based on spherical harmonics (see, for example, \cite{ABR}) and the Hankel transform (see \cite{EK}).

We now define the asymptotic profile up to order \(n\):
\begin{equation}\label{asymptoticprofiledef}
    A_n(x,t):=\sum_{j=0}^n \mathcal{P}_j(x,t),
\end{equation}
where for $j=0,1,\dots,n$
\begin{equation}\label{pndef}
\mathcal{P}_j(x,t):=M_{\nu,j}(u_0)t^{-1-\nu-j}|x|^{\nu-\mu}e^{-\frac{|x|^2}{4t}}P_j\left(\frac{|x|^2}{t}\right),
\end{equation}
with
\[P_j(s):=\sum_{m=0}^{j}\frac{(-1)^{j-m}}{2^{2(\nu+j+m)+1}m!(j-m)!\Gamma(m+\nu+1)}
s^m.\]

\section{Main results}
\begin{theorem}\label{asymptoticprofilethm}
Let \(u_0:(0,\infty)\to\mathbb R\) be radial such that
\begin{equation}
    \int_{\mathbb R^N}
|x|^{\nu-\mu}\bigl(1+|x|^{2n}\bigr)|u_0(|x|)|\,dx<\infty.
\end{equation}
and Let \(u(\cdot,t)=S_\lambda(t)u_0\) be the mild solution to \eqref{problem} given by the radial kernel formula \eqref{radialkerneldef}, that is,
\[
u(x,t)
=
\int_{\mathbb R^N}
K(x,y,t)u_0(y)\,dy.
\] 
Then
\[
\lim_{t\to\infty}
t^{\gamma_p+n}
\left\|
u(\cdot,t)
-
A_n(\cdot,t)
\right\|_{L^p(\mathbb R^N)}
=0
\]
for all $p\in\left[\frac{2(N-1)}{N+2\nu+4n},\frac{N}{\mu-\nu}\right)$. 
\end{theorem}
\begin{remark}\label{sharpness}
The decay rate \(\gamma_p+n\) is sharp under the assumption \(u_0\in L^1_{\nu,n}(\mathbb R^N)\), in the sense that one cannot in general
obtain a faster decay for the remainder after subtracting \(A_n\) without
imposing additional assummptions.
\end{remark}
\begin{remark}
The range of admissible exponents \(p\) in the above result is not global. The upper bound on \(p\) comes from the integrability of the asymptotic profiles near the singularity \(x=0\), whereas the lower bound arises in the estimates of the large-\(z=\frac{|x||y|}{2t}\) part of the heat semigroup. However, for \(n\) sufficiently large, namely $
n\geq\frac{N-2-2\nu}{4}, $
then $
\frac{2(N-1)}{N+2\nu+4n}\leq 1. $
Consequently, this lower bound becomes irrelevant for \(1\le p<\infty\).
\end{remark}
\begin{remark}[The classical heat equation]
If \(\lambda=0\), then \(\nu=\mu\).
The leading $n=0$ profile becomes 
\[
A(x,t)
=
c_\mu M(u_0)t^{-\frac N2}
\exp\left(-\frac{|x|^2}{4t}\right),
\]
where $\displaystyle M(u_0)
=
\int_{\mathbb R^N}u_0(|x|)\,dx. $
Since $
c_\mu
=
\frac{1}{2^{2\mu+1}\Gamma(\mu+1)}
=
\frac{1}{2^{N-1}\Gamma(N/2)}, $
this agrees with the usual Gaussian heat profile, up to the normalization
convention for the radial kernel. Thus Theorem~\ref{asymptoticprofilethm}
recovers the classical asymptotic behaviour
\[
\lim_{t\to\infty}
t^{\frac N2\left(1-\frac1p\right)}
\left\|
u(\cdot,t)-M(u_0)G(\cdot,t)
\right\|_{L^p(\mathbb R^N)}
=0.
\]
\end{remark}
\begin{remark}
For $n=0$, $\lambda=\dfrac{(N-2)^2}{4}$, we recover, in the radial setting  the \(L^2\)-asymptotic behaviour obtained by Vasquez and Zuazua in \cite[Theorem 10.3]{VZ}. Indeed, the asymptotic profile in Theorem \ref{asymptoticprofilethm} becomes
\[
A(x,t)=
\frac12 M_0(u_0)t^{-1}|x|^{-\mu}
\exp\left(-\frac{|x|^2}{4t}\right),
\qquad
M_0(u_0)
=
\int_{\mathbb R^N}|x|^{-\mu}u_0(|x|)\,dx.
\]
Moreover,
\[
\gamma_p=\frac{N+2}{4}-\frac{N}{2p}.
\]
In particular, for \(p=2\), $
\gamma_2=\frac12, $
and therefore
\[
\lim_{t\to\infty}
t^{1/2}
\left\|
u(\cdot,t)
-
A(x,t)\right\|_{L^2(\mathbb R^N)}
=0.
\]
\end{remark}

For non-radial data, the situation is more delicate. Since the potential
\(|x|^{-2}\) is radial, the operator can be decomposed by means of spherical
harmonics (see, for example, \cite{ABR}). Each angular term leads to a Bessel-type radial operator with a
different parameter
\[
    \nu_\ell
    =
    \sqrt{(\ell+\mu)^2-\lambda},
    \qquad \ell\in\mathbb N_0.
\]
Thus, if \(x=r\omega\) and \(y=\rho\eta\), the full kernel for \eqref{problem} is given by
\begin{equation}\label{nonradialheatkerneldef}
K(x,y,t)=
\frac{1}{2t}(r\rho)^{-\mu}
\exp\!\left(-\frac{r^2+\rho^2}{4t}\right)
\sum_{\ell=0}^{\infty}  
I_{\nu_\ell}\left(\frac{r\rho}{2t}\right)
\sum_{m=1}^{d_\ell}Y_{\ell,m}(\omega)Y_{\ell,m}(\eta), 
\end{equation}
where \(d_\ell\) denotes the dimension of the eigenspace of
\(-\Delta_{\mathbb S^{N-1}}\) associated with the eigenvalue
\(\ell(\ell+N-2)\) and \(Y_{\ell,m}\) are the associated eigenfunctions. The complete derivation of $K(x,y,t)$ can be found in the Appendix, Proposition \ref{nonradialkernel}.

The spherical harmonic decomposition also allows us to extend the asymptotic
profile result beyond the radial class. Since each angular term evolves according
to a radial Hardy heat equation with parameter
\[
    \nu_\ell=\sqrt{(\ell+\mu)^2-\lambda},
\]
the radial asymptotic expansion can be applied term by term. This gives a
non-radial profile obtained by summing the corresponding angular profiles. 

We state the asymptotic result for initial data with finite angular expansion, where no summability difficulty occurs. We then give an infinite-dimensional version under
an additional summability assumption on the spherical harmonic coefficients.

\begin{theorem}\label{nonradial-finite}
Let \(x=r\omega\), with \(r>0\) and \(\omega\in\mathbb S^{N-1}\). Assume that
\(u_0\) has a finite spherical harmonic expansion
\[
u_0(r,\omega)
=
\sum_{\ell=0}^{L}
\sum_{m=1}^{d_\ell}
u_{0,\ell,m}(r)Y_{\ell,m}(\omega),
\]
where, for each \(\ell=0,\dots,L\) and \(m=1,\dots,d_\ell\),
\[
\int_{\mathbb R^N}
|x|^{\nu_\ell-\mu}
\bigl(1+|x|^{2n}\bigr)
|u_{0,\ell,m}(|x|)|\,dx<\infty.
\]
Let \(u(\cdot,t)=S_\lambda(t)u_0\) be the mild solution given by the kernel \eqref{nonradialheatkerneldef}. Then, for every $p\in\left[\frac{2(N-1)}{N+2\nu+4n},\frac{N}{\mu-\nu}\right)$
one has
\[
\lim_{t\to\infty}
t^{\gamma_p+n}
\left\|
u(\cdot,t)-A_{n,L}(\cdot,t)
\right\|_{L^p(\mathbb R^N)}
=0,
\]
where

\[
A_{n,L}(x,t)
=
\sum_{\ell=0}^{L}
\sum_{m=1}^{d_\ell}
\sum_{j=0}^{n}
\mathcal P_{\ell,m,j}(x,t)Y_{\ell,m}(\omega),
\]
with
\[
\mathcal P_{\ell,m,j}(x,t)
=
M_{\nu_\ell,j}^{\ell,m}
t^{-1-\nu_\ell-j}
|x|^{\nu_\ell-\mu}
e^{-\frac{|x|^2}{4t}}
P_j^{(\nu_\ell)}
\left(\frac{|x|^2}{t}\right),
\]
and
\[
M_{\nu_\ell,j}^{\ell,m}
=
\int_{\R^N}
|x|^{\nu_\ell-\mu+2j}
u_{0,\ell,m}(|x|)dx.
\]
where \(P_j^{(\alpha)}\) is defined by
\[
P_j^{(\alpha)}(s)
=
\sum_{q=0}^{j}
\frac{(-1)^{j-q}}
{2^{2(\alpha+j+q)+1}q!(j-q)!\Gamma(q+\alpha+1)}
s^q.
\]
\end{theorem}

\begin{theorem}\label{nonradial-infinite} Let \(x=r\omega\), with \(r>0\) and \(\omega\in\mathbb S^{N-1}\).
Assume that \(u_0\) admits the spherical harmonic expansion
\[
u_0(r,\omega)
=
\sum_{\ell=0}^{\infty}
\sum_{m=1}^{d_\ell}
u_{0,\ell,m}(r)Y_{\ell,m}(\omega).
\]
For each \((\ell,m)\), assume that
\[
\|u_{0,\ell,m}\|_{L^1_{\nu_\ell,n+1}}
:=
\int_{\R^N}
|x|^{\nu_\ell-\mu}
\bigl(1+|x|^{2n+2}\bigr)
|u_{0,\ell,m}(|x|)|dx
<\infty.
\]
Assume moreover that
\[
\sum_{\ell=0}^{\infty}
\sum_{m=1}^{d_\ell}
\|Y_{\ell,m}\|_{L^p(\mathbb S^{N-1})}
C_{\ell,n,p}
\|u_{0,\ell,m}\|_{L^1_{\nu_\ell,n+1}}
<\infty.
\]
Let \(u(\cdot,t)=S_\lambda(t)u_0\) be the mild solution defined by the
non-radial kernel \eqref{nonradialheatkerneldef}. Then, for every $p\in\left[\frac{2(N-1)}{N+2\nu+4n+4},\frac{N}{\mu-\nu}\right)$
one has
\[
\lim_{t\to\infty}
t^{\gamma_p+n}
\left\|
u(\cdot,t)-A_n^\infty(\cdot,t)
\right\|_{L^p(\mathbb R^N)}
=0,
\]
where
\[
A_n^\infty(x,t)
=
\sum_{\ell=0}^{\infty}
\sum_{m=1}^{d_\ell}
\sum_{j=0}^{n}
\mathcal P_{\ell,m,j}(x,t)Y_{\ell,m}(\omega),
\]
with
\[
\mathcal P_{\ell,m,j}(x,t)
=
M_{\nu_\ell,j}^{\ell,m}
t^{-1-\nu_\ell-j}
|x|^{\nu_\ell-\mu}
e^{-\frac{|x|^2}{4t}}
P_j^{(\nu_\ell)}
\left(\frac{|x|^2}{t}\right),
\]
and
\[
M_{\nu_\ell,j}^{\ell,m}
=
\int_{\R^N}
|x|^{\nu_\ell-\mu+2j}
u_{0,\ell,m}(|x|)dx.
\]
where \(P_j^{(\alpha)}\) is defined by
\[
P_j^{(\alpha)}(s)
=
\sum_{q=0}^{j}
\frac{(-1)^{j-q}}
{2^{2(\alpha+j+q)+1}q!(j-q)!\Gamma(q+\alpha+1)}
s^q.
\]
\end{theorem}

\begin{remark}
In the case \(p=2\), the summability assumption in Theorem \ref{nonradial-infinite} can be weakened because the spherical harmonics are orthonormal in \(L^2(\mathbb S^{N-1})\). Indeed, if
\[
R_{\ell,m}(r,t)
=
u_{\ell,m}(r,t)
-
\sum_{j=0}^{n}\mathcal P_{\ell,m,j}(r,t),
\]
then Parseval's identity gives
\[
\left\|
\sum_{\ell=0}^{\infty}
\sum_{m=1}^{d_\ell}
R_{\ell,m}(r,t)Y_{\ell,m}(\omega)
\right\|_{L^2(\mathbb R^N)}^2
=
\sum_{\ell=0}^{\infty}
\sum_{m=1}^{d_\ell}
\|R_{\ell,m}(\cdot,t)\|_{L^2(\R^N)}^2.
\]
Therefore, instead of the absolute summability condition used for general \(p\),
it is enough to assume the square-summability condition
\[
\sum_{\ell=0}^{\infty}
\sum_{m=1}^{d_\ell}
C_{\ell,n,2}^2
\|u_{0,\ell,m}\|_{1,\nu_\ell,n+1}^2
<\infty.
\]
Note that the lower bound on $p$ is always smaller than $2$.
\end{remark}
\section{Asymptotic Profile of solutions}
\begin{lemma}\label{lemma0.2}
Let \(1\le p<\frac{N}{\mu-\nu}\).
Then, for every $j=0,1,\dots n$,
\[
\|\mathcal P_j(\cdot,t)\|_{L^p(\mathbb R^N)}
=
C_{j,p,\nu}|M_{\nu,j}(u_0)|t^{-\gamma_p-j},
\]
where
\[
C_{j,p,\nu}
:=
\left(
\int_{\mathbb R^N}
|z|^{p(\nu-\mu)}
e^{-\frac{p|z|^2}{4}}
\left|P_j(|z|^2)\right|^p
\,dz
\right)^{1/p}
>0.
\]
Consequently, \(A_n(\cdot,t)\in L^p(\mathbb R^N)\) and $\|A_n(\cdot,t)\|_{L^p(\R^N)}\leq C\cdot M_{\nu,j}(|u_0|)t^{-\gamma_p-n}$.
\end{lemma}
\begin{proof}
By definition \eqref{pndef}.
\[
\mathcal P_j(x,t)
=
M_{\nu,j}(u_0)
t^{-1-\nu-j}
|x|^{\nu-\mu}
e^{-\frac{|x|^2}{4t}}
P_j\left(\frac{|x|^2}{t}\right).
\]
Therefore
\[
\begin{aligned}
\|\mathcal P_j(\cdot,t)\|_{L^p(\R^N)}^p
&=
|M_{\nu,j}(u_0)|^p
t^{-p(1+\nu+j)}
\int_{\mathbb R^N}
|x|^{p(\nu-\mu)}
e^{-\frac{p|x|^2}{4t}}
\left|
P_j\left(\frac{|x|^2}{t}\right)
\right|^p
\,dx.
\end{aligned}
\]
Changing variables \(x=\sqrt t\,z\), we obtain
\[
\begin{aligned}
\|\mathcal P_j(\cdot,t)\|_{L^p(\R^N)}^p
&=
|M_{\nu,j}(u_0)|^p
t^{-p(1+\nu+j)}
t^{\frac{p(\nu-\mu)}2}
t^{\frac N2}
\\
&\qquad\times
\int_{\mathbb R^N} |z|^{p(\nu-\mu)} e^{-\frac{p|z|^2}{4}}|P_j(|z|^2)|^p\,dz.
\end{aligned}
\]
The condition \(p(\nu-\mu)+N>0\) guarantees integrability near the origin and while the Gaussian term gives decay at infinity. It is strictly positive since \(P_n\) is not
identically zero. Taking the \(p\)-th root and using
\[
\gamma_p:=\frac N2\left(1-\frac1p\right)
-\frac{\mu-\nu}{2}=
1+\frac{\mu+\nu}{2}-\frac N{2p},
\]
we get
\[
\|\mathcal P_j(\cdot,t)\|_{L^p(\R^N)}
=
C_{j,p,\nu}|M_{\nu,j}(u_0)|t^{-\gamma_p-j}.
\]
Summing over $j=0,1,\dots n$ and using $|M_{\nu,j}(u_0)|\leq M_{\nu,j}(|u_0|)$ gives the final result
\end{proof}

The asymptotic profile we present arises from the small $z$ (with \(z=\frac{|x||y|}{2t}\)) expansion of the kernel \eqref{radialkerneldef} in the following way: we first expand the Bessel function \(I_\nu(z)\) and then the Gaussian factor
\(e^{-|y|^2/(4t)}\), each term is indexed by two integers \(m\) and \(\ell\).
Their total order \(j=m+\ell\) produces the weighted moment $M_{\nu,j}(u_0)$ defined in \eqref{momentsdef}
Collecting all terms with the same index \(j\) gives the profile \(\mathcal{P}_j\). The
remaining terms are the Bessel remainder, the Gaussian remainder, the cut-off
error, and the large-\(z\) contribution. We estimate each of these remainder terms separately.

\begin{prop}\label{asymptoticbound}
Let \(u_0:(0,\infty)\to\mathbb R\) be radial such that
\begin{equation}\label{momentassumption}
    \int_{\mathbb R^N}
|x|^{\nu-\mu}\bigl(1+|x|^{2n+2}\bigr)|u_0(|x|)|\,dx<\infty.
\end{equation}

and let $u(x,t)$ be the heat semigroup associated to the problem \eqref{problem} with initial data $u_0$. Then there exists a constant $C>0$ such that, for $p\in\left[\dfrac{2(N-1)}{N+4+2\nu+4n},\dfrac{N}{\mu-\nu}\right)$, we have  
\[
t^{\gamma_p+n}
\left\|
u(\cdot,t)
-
A_n(\cdot,t)
\right\|_{L^p(\mathbb R^N)}
\leq C\cdot M_{\nu,n+1}(|u_0|),\qquad t\geq 1
\]
\end{prop}

\begin{proof} Set $
z=\frac{|x||y|}{2t}.$ Let \(\chi\in C_c^\infty([0,\infty))\) satisfy
\[
0\le \chi\le 1,
\qquad
\chi(s)=1 \text{ for }0\le s\le 1,
\qquad
\chi(s)=0 \text{ for }s\ge 2.
\]
We decompose
\[
u(x,t)=u_{\rm small}(x,t)+u_{\rm large}(x,t),
\]
where
\[
u_{\rm small}(x,t)
=
\int_{\mathbb R^N}
\chi(z)K(x,y,t)u_0(y)\,dy
\]
and
\[
u_{\rm large}(x,t)
=
\int_{\mathbb R^N}
(1-\chi(z))K(x,y,t)u_0(y)\,dy.
\]
The asymptotic profiles are obtained from \(u_{\rm small}\). The term
\(u_{\rm large}\) will be estimated separately and shown to be of the same
order as the remainder.

\textbf{Step 1: Estimating the Bessel remainder}

On the support of \(\chi\), we have \(0\le z\le 2\). Hence we may use the
small-argument expansion \cite{W} for the bessel function $I_{\nu}$:
\[
I_\nu(z)
=
\sum_{m=0}^{n}
\frac{1}{m!\Gamma(m+\nu+1)}
\left(\frac z2\right)^{2m+\nu}
+
R_{I,n}(z),
\]
where , for \(0\le z\le 2\),
\[
|R_{I,n}(z)|
\le
C_n z^{2n+\nu+2}.
\]
Since $
\frac z2=\frac{|x||y|}{4t}, $
substituting this expansion into the kernel gives
\[
\begin{aligned}
K(x,y,t)
&=
\sum_{m=0}^{n}
a_{m,\nu}
t^{-1-\nu-2m}
|x|^{\nu-\mu+2m}
|y|^{\nu-\mu+2m}
e^{-\frac{|x|^2}{4t}}
e^{-\frac{|y|^2}{4t}}
\\
&\qquad
+
\mathcal R_{I,n}(x,y,t),
\end{aligned}
\]
where
\[
a_{m,\nu}
:=
\frac{1}{2^{2\nu+4m+1}m!\Gamma(m+\nu+1)}
\]
and
\[
\mathcal R_{I,n}(x,y,t)
:=
\frac1{2t}
(|x||y|)^{-\mu}
e^{-\frac{|x|^2+|y|^2}{4t}}
R_{I,n}\left(\frac{|x||y|}{2t}\right).
\]
The bound on \(R_{I,n}\) gives
\[
|\mathcal R_{I,n}(x,y,t)|
\le
C
t^{-3-\nu-2n}
|x|^{\nu-\mu+2n+2}
|y|^{\nu-\mu+2n+2}
e^{-\frac{|x|^2+|y|^2}{4t}}.
\]

Therefore the contribution of the Bessel remainder to the solution is
\[
\mathcal E_{I,n}(x,t)
:=
\int_{\mathbb R^N}
\chi(z)
\mathcal R_{I,n}(x,y,t)u_0(y)\,dy,
\]

Using the pointwise bound for \(\mathcal R_{I,n}\), we obtain
\[
\begin{aligned}
|\mathcal E_{I,n}(x,t)|
&\le
C
t^{-3-\nu-2n}
|x|^{\nu-\mu+2n+2}
e^{-\frac{|x|^2}{4t}}
\\
&\qquad \times
\int_{\mathbb R^N}
\chi(z)
|y|^{\nu-\mu+2n+2}
e^{-\frac{|y|^2}{4t}}
|u_0(y)|\,dy .
\end{aligned}
\]
Since \(0\le \chi\le 1\) and \(e^{-|y|^2/(4t)}\le 1\), it follows that
\[
|\mathcal E_{I,n}(x,t)|
\le
C
t^{-3-\nu-2n}
|x|^{\nu-\mu+2n+2}
e^{-\frac{|x|^2}{4t}}
\int_{\mathbb R^N}
|y|^{\nu-\mu+2n+2}
|u_0(y)|\,dy .
\]
By the moment assumption \eqref{momentassumption}, we get
\[
\int_{\mathbb R^N}
|y|^{\nu-\mu+2n+2}
|u_0(y)|\,dy<\infty,
\]
thus
\[
|\mathcal E_{I,n}(x,t)|
\le
C\cdot M_{\nu,n+1}(|u_0|)\cdot
t^{-3-\nu-2n}
|x|^{\nu-\mu+2n+2}
e^{-\frac{|x|^2}{4t}}.
\]

Taking the \(L^p(\mathbb R^N)\)-norm gives
\[
\|\mathcal E_{I,n}(t)\|_{L^p(\R^N)}
\le
C\cdot M_{\nu,n+1}(|u_0|)
t^{-3-\nu-2n}
\left\|
|x|^{\nu-\mu+2n+2}
e^{-\frac{|x|^2}{4t}}
\right\|_{L^p(\R^N)}.
\]
Using Lemma \ref{gammaint}, we have
\[
\left\|
|x|^{\nu-\mu+2n+2}
e^{-\frac{|x|^2}{4t}}
\right\|_{L^p(\R^N)}
=
C
t^{\frac{\nu-\mu+2n+2}{2}+\frac{N}{2p}},
\]
provided that $p(\nu-\mu+2n+2)+N>0. $
This condition is guaranteed by the upper $p$ bound $p<\frac{N}{\mu-\nu}$.
Therefore
\[
\|\mathcal E_{I,n}(t)\|_{L^p(\R^N)}
\le
C\cdot M_{\nu,n+1}(|u_0|)\cdot
t^{-3-\nu-2n
+\frac{\nu-\mu+2n+2}{2}
+\frac{N}{2p}}.
\]
Since $\mu=(N-2)/2$, the exponent simplifies as
\[
-3-\nu-2n
+\frac{\nu-\mu+2n+2}{2}
+\frac{N}{2p}
=
-\gamma_p-n-1,
\]
where $\gamma_p$ is defined in \eqref{gammapdef}.

Hence
\[
\|\mathcal E_{I,n}(t)\|_{L^p(\mathbb R^N)}
\le
C\cdot M_{\nu,n+1}(|u_0|)\cdot t^{-\gamma_p-n-1}.
\]
The remaining part of the finite Bessel expansion is therefore
\begin{equation}\label{besselremainder}
    u_{small}(x,t)-\mathcal{E}_{I,n}(x,t)=\sum_{m=0}^{n}
a_{m,\nu}
t^{-1-\nu-2m}
|x|^{\nu-\mu+2m}
e^{-\frac{|x|^2}{4t}}
\int_{\mathbb R^N}
\chi(z)|y|^{\nu-\mu+2m}
e^{-\frac{|y|^2}{4t}}
u_0(y)\,dy.
\end{equation}

\textbf{Step 2: Estimating the Gaussian error}

We now expand the Gaussian factor depending on \(y\). For each fixed
\(m\in\{0,\dots,n\}\), we write
\[
e^{-\frac{|y|^2}{4t}}
=
\sum_{\ell=0}^{n-m}
\frac{(-1)^\ell}{4^\ell \ell!}
\frac{|y|^{2\ell}}{t^\ell}
+
R_{G,n-m}(y,t),
\]
where
\[
|R_{G,n-m}(y,t)|
\le
C
\left(\frac{|y|^2}{t}\right)^{n-m+1}.
\]
The error produced by the exponential remainder is
\[
\mathcal E_{G,n}(x,t)
:=
\sum_{m=0}^{n}
a_{m,\nu}
t^{-1-\nu-2m}
|x|^{\nu-\mu+2m}
e^{-\frac{|x|^2}{4t}}
\int_{\mathbb R^N}
\chi(z)|y|^{\nu-\mu+2m}
R_{G,n-m}(y,t)
u_0(y)\,dy.
\]

Using the bound on \(R_{G,n-m}\) and the fact that \(0\le \chi\le 1\), we obtain
\[
\begin{aligned}
|\mathcal E_{G,n}(x,t)|
&\le
C
\sum_{m=0}^{n}
t^{-1-\nu-2m}
|x|^{\nu-\mu+2m}
e^{-\frac{|x|^2}{4t}}
\\
&\qquad \times
\int_{\mathbb R^N}
|y|^{\nu-\mu+2m}
\left(\frac{|y|^2}{t}\right)^{n-m+1}
|u_0(y)|\,dy .
\end{aligned}
\]
Combining the powers of \(|y|\), we have
\[
|\mathcal E_{G,n}(x,t)|
\le
C
\sum_{m=0}^{n}
t^{-2-\nu-n-m}
|x|^{\nu-\mu+2m}
e^{-\frac{|x|^2}{4t}}
\int_{\mathbb R^N}
|y|^{\nu-\mu+2n+2}|u_0(y)|\,dy.
\]
By the moment assumption \eqref{momentassumption}
\[
\int_{\mathbb R^N}
|y|^{\nu-\mu+2n+2}|u_0(y)|\,dy<\infty,
\]
we get
\[
|\mathcal E_{G,n}(x,t)|
\le
C\cdot M_{\nu,n+1}(|u_0|)\cdot
\sum_{m=0}^{n}
t^{-2-\nu-n-m}
|x|^{\nu-\mu+2m}
e^{-\frac{|x|^2}{4t}}.
\]

Taking the \(L^p(\mathbb R^N)\)-norm, we obtain
\[
\|\mathcal E_{G,n}(t)\|_{L^p(\R^N)}
\le
C\cdot M_{\nu,n+1}(|u_0|)\cdot
\sum_{m=0}^{n}
t^{-2-\nu-n-m}
\left\|
|x|^{\nu-\mu+2m}
e^{-\frac{|x|^2}{4t}}
\right\|_{L^p(\R^N)}.
\]
By Lemma~\ref{gammaint}, with $
a=\nu-\mu+2m,$ $c=\frac14, $ we have
\[
\left\|
|x|^{\nu-\mu+2m}
e^{-\frac{|x|^2}{4t}}
\right\|_{L^p(\R^N)}
=
C
t^{\frac{\nu-\mu+2m}{2}+\frac{N}{2p}},
\]
provided that $p(\nu-\mu+2m)+N>0. $
This condition is implied by $p<\dfrac{N}{\mu-\nu}$ since $m\geq 0$.
Hence
\[
\|\mathcal E_{G,n}(t)\|_{L^p(\R^N)}
\le
C\cdot M_{\nu,n+1}(|u_0|)\cdot 
\sum_{m=0}^{n}
t^{-2-\nu-n-m
+\frac{\nu-\mu+2m}{2}
+\frac{N}{2p}}.
\]
Since $\mu=(N-2)/2$, for every \(m=0,\dots,n\), the exponent simplifies to
\[
-2-\nu-n-m
+
\frac{\nu-\mu+2m}{2}
+
\frac{N}{2p}
=
-\gamma_p-n-1,
\]
where $\gamma_p$ is defined in \eqref{gammapdef}. Therefore
\[
\|\mathcal E_{G,n}(t)\|_{L^p(\R^N)}
\le
C\cdot M_{\nu,n+1}(|u_0|)\cdot
\sum_{m=0}^{n}
t^{-\gamma_p-n-1}.
\]
Since the sum has only finitely many terms, we conclude that
\[
\|\mathcal E_{G,n}(t)\|_{L^p(\mathbb R^N)}
\le
C\cdot M_{\nu,n+1}(|u_0|)\cdot t^{-\gamma_p-n-1}.
\]

\textbf{Step 3: Estimating the cutoff error}

Using \eqref{besselremainder} and the Gaussian expansion and subtracting the  error yields
\begin{equation}
\begin{split}
    u_{small}(x,t)&-\mathcal{E}_{I,n}(x,t)-\mathcal{E}_{G,n}(x,t)\\
&=
\sum_{m=0}^{n}
\sum_{\ell=0}^{n-m}
a_{m,\nu}
\frac{(-1)^\ell}{4^\ell\ell!}
t^{-1-\nu-2m-\ell}
|x|^{\nu-\mu+2m}
e^{-\frac{|x|^2}{4t}}
\int_{\mathbb R^N}
\chi(z)|y|^{\nu-\mu+2m+2\ell}
u_0(y)\,dy.
\end{split}
\end{equation}

We now obtain the cut-off error using $\chi=1+(\chi-1)$. Let
\begin{equation}\label{Andef}
   A_n(x,t)
:=
\sum_{m=0}^{n}
\sum_{\ell=0}^{n-m}
a_{m,\nu}
\frac{(-1)^\ell}{4^\ell\ell!}
t^{-1-\nu-2m-\ell}
|x|^{\nu-\mu+2m}
e^{-\frac{|x|^2}{4t}}
\int_{\mathbb R^N}
|y|^{\nu-\mu+2m+2\ell}
u_0(y)\,dy. 
\end{equation}

and
\[
\mathcal E_{\chi,n}(x,t):=
\sum_{m=0}^{n}
\sum_{\ell=0}^{n-m}
a_{m,\nu}
\frac{(-1)^\ell}{4^\ell\ell!}
t^{-1-\nu-2m-\ell}
|x|^{\nu-\mu+2m}
e^{-\frac{|x|^2}{4t}}
\int_{\mathbb R^N}
(\chi(z)-1)|y|^{\nu-\mu+2m+2\ell}
u_0(y)\,dy.
\]
We will later show that $A_n$ is precisely the profile defined in \eqref{Andef}. Let $
j:=m+\ell. $
Then \(0\le j\le n\), and
\[
\nu-\mu+2m+2\ell=\nu-\mu+2j.
\]
Since $\chi(z)-1\neq 0$ only for \(\{z=\frac{|x||y|}{2t}\ge1\}\), we may write, for such $z$
\[
1
\le
C\left(\frac{|x||y|}{t}\right)^{2(n+1-j)},\qquad j=0,1,\dots n.
\]
Therefore
\[
|\chi(z)-1|
\le \mathbf 1_{\{z\ge1\}}\leq
C
\left(\frac{|x||y|}{t}\right)^{2(n+1-j)}
\mathbf 1_{\{z\ge1\}}.
\]
Hence
\[
\begin{aligned}
&\left|
\int_{\mathbb R^N}
(\chi(z)-1)
|y|^{\nu-\mu+2j}
u_0(y)\,dy
\right|
\\
&\qquad\le
C
t^{-2(n+1-j)}
|x|^{2(n+1-j)}
\int_{\mathbb R^N}
|y|^{\nu-\mu+2n+2}
|u_0(y)|\,dy.\\
&\qquad =
C
t^{-2(n+1-j)}
|x|^{2(n+1-j)} M_{\nu,n+1}(|u_0|).
\end{aligned}
\]
Consequently,
\[
\begin{aligned}
|\mathcal E_{\chi,n}(x,t)|
&\le
C\cdot M_{\nu,n+1}(|u_0|)\cdot
\sum_{m=0}^{n}
\sum_{\ell=0}^{n-m}
t^{-1-\nu-2m-\ell}\cdot
t^{-2(n+1-j)}
\\
&\qquad\qquad\times
|x|^{\nu-\mu+2m+2(n+1-j)}
e^{-\frac{|x|^2}{4t}}.
\end{aligned}
\]
Taking the \(L^p(\mathbb R^N)\)-norm gives
\[
\begin{aligned}
\|\mathcal E_{\chi,n}(t)\|_{L^p(\R^N)}
&\le
C\cdot M_{\nu,n+1}(|u_0|)\cdot
\sum_{m=0}^{n}
\sum_{\ell=0}^{n-m}
t^{-1-\nu-2m-\ell-2(n+1-j)}
\\
&\qquad\qquad\times
\left\|
|x|^{\nu-\mu+2m+2(n+1-j)}
e^{-\frac{|x|^2}{4t}}
\right\|_{L^p(\R^N)}.
\end{aligned}
\]
By Lemma \ref{gammaint}, with $a=\nu-\mu+2m+2(n+1-j), \ 
c=\frac14,$
we have
\[
\left\|
|x|^{\nu-\mu+2m+2(n+1-j)}
e^{-\frac{|x|^2}{4t}}
\right\|_{L^p(\R^N)}
=
C
t^{\frac{\nu-\mu+2m+2(n+1-j)}{2}+\frac N{2p}},
\]
provided
\[
p\bigl(\nu-\mu+2m+2(n+1-j)\bigr)+N>0.
\]
This condition is guaranteed by the assumption $p<\frac{N}{\mu-\nu}$.
Therefore
\[
\begin{aligned}
\|\mathcal E_{\chi,n}(t)\|_{L^p(\R^N)}
&\le
C\cdot M_{\nu,n+1}(|u_0|)\cdot
\sum_{m=0}^{n}
\sum_{\ell=0}^{n-m}
t^{-1-\nu-2m-\ell-2(n+1-j)}\cdot t^{\frac{\nu-\mu+2m+2(n+1-j)}{2}+\frac N{2p}}.
\end{aligned}
\]
Using \(j=m+\ell\), the exponent becomes $-\gamma_p-n-1, $
where is defined in \eqref{gammapdef}. Thus each term in the finite double sum is bounded by
\[
C t^{-\gamma_p-n-1}.
\]
Hence
\[
\|\mathcal E_{\chi,n}(t)\|_{L^p(\mathbb R^N)}
\le
C \cdot M_{\nu,n+1}(|u_0|)\cdot t^{-\gamma_p-n-1}.
\]
\textbf{Step 4: Estimating $u_{\text{large}}$.} On the support of \(1-\chi(z)\), we have \(z\ge 1\). We use the
large-argument estimate
\[
I_\nu(z)\le C z^{-1/2}e^z,\qquad z\ge 1.
\]
Therefore
\[
\begin{aligned}
|(1-\chi(z))K(x,y,t)|
&\le
C t^{-1}(|x||y|)^{-\mu}
e^{-\frac{|x|^2+|y|^2}{4t}}
\left(\frac{|x||y|}{2t}\right)^{-1/2}
e^{\frac{|x||y|}{2t}}
\mathbf 1_{\{|x||y|\ge 2t\}}      \\
&\le
C t^{-1/2}
(|x||y|)^{-\mu-\frac12}
e^{-\frac{(|x|-|y|)^2}{4t}}
\mathbf 1_{\{|x||y|\ge 2t\}} .
\end{aligned}
\]
Since $\mu=\frac{N-2}{2}$, we have $
-\mu-\frac12=-\frac{N-1}{2}. $
Hence
\[
|(1-\chi(z))K(x,y,t)|
\le
C t^{-1/2}
|x|^{-\frac{N-1}{2}}
|y|^{-\frac{N-1}{2}}
e^{-\frac{(|x|-|y|)^2}{4t}}
\mathbf 1_{\{|x||y|\ge 2t\}} .
\]
Consequently,
\[
|u_{\rm large}(x,t)|
\le
C t^{-1/2}
|x|^{-\frac{N-1}{2}}
\int_{\mathbb R^N}
|y|^{-\frac{N-1}{2}}
e^{-\frac{(|x|-|y|)^2}{4t}}
\mathbf 1_{\{|x||y|\ge 2t\}}
|u_0(|y|)|\,dy .
\]

We now estimate this in \(L^p(\mathbb R^N)\). By Minkowski's inequality,
for \(1\le p<\infty\),
\begin{equation}\label{eq11}
    \begin{aligned}
\|u_{\rm large}(t)\|_{L^p(\R^N)}
&\le
C t^{-1/2}
\int_{\mathbb R^N}
|y|^{-\frac{N-1}{2}}
|u_0(y)|
\\
&\qquad\qquad\times
\left\|
|x|^{-\frac{N-1}{2}}
e^{-\frac{(|x|-|y|)^2}{4t}}
\mathbf 1_{\{|x||y|\ge 2t\}}
\right\|_{L^p_x}
dy .
\end{aligned}
\end{equation}

For fixed \(y\), using polar coordinates in the \(x\)-variable,
\[
\begin{aligned}
&\left\|
|x|^{-\frac{N-1}{2}}
e^{-\frac{(|x|-|y|)^2}{4t}}
\mathbf 1_{\{|x||y|\ge 2t\}}
\right\|_{L^p_x}^p  \\
&\qquad =
C
\int_{2t/|y|}^{\infty}
r^{-\frac{(N-1)p}{2}+N-1}
e^{-\frac{p(r-|y|)^2}{4t}}
\,dr .
\end{aligned}
\]
Now put $
r=\sqrt t\,s,
\qquad
|y|=\sqrt t\,\eta . $
Then
\[
\begin{aligned}
&\left\|
|x|^{-\frac{N-1}{2}}
e^{-\frac{(|x|-|y|)^2}{4t}}
\mathbf 1_{\{|x||y|\ge 2t\}}
\right\|_{L^p_x}  \\
&\qquad =
C t^{-\frac{N-1}{4}+\frac N{2p}}
\left(
\int_{2/\eta}^{\infty}
s^{-\frac{(N-1)p}{2}+N-1}
e^{-\frac p4(s-\eta)^2}
\,ds
\right)^{1/p}.
\end{aligned}
\]

Using Lemma \ref{localization} with $c_0=2$ and $\beta=-\frac{(N-1)p}{2}+N-1$, for every \(M>0\),
\[
\left(
\int_{2/\eta}^{\infty}
s^{-\frac{(N-1)p}{2}+N-1}
e^{-\frac p4(s-\eta)^2}
\,ds
\right)^{1/p}
\le C_M\eta^M,
\qquad 0<\eta\le1,
\]
and
\[
\left(
\int_{2/\eta}^{\infty}
s^{-\frac{(N-1)p}{2}+N-1}
e^{-\frac p4(s-\eta)^2}
\,ds
\right)^{1/p}
\le
C\eta^{-\frac{N-1}{2}+\frac{N-1}{p}},
\qquad \eta\ge1.
\]
Since $|y|=\sqrt{t}\eta$, we get
\begin{equation}\label{idk1}
\left\|
|x|^{-\frac{N-1}{2}}
e^{-\frac{(|x|-|y|)^2}{4t}}
\mathbf 1_{\{|x||y|\ge 2t\}}
\right\|_{L^p_x}
\le
C
\begin{cases}
t^{-\frac{N-1}{4}+\frac N{2p}-\frac M2}|y|^M,
& 0<|y|\le \sqrt t, \\[2mm]
t^{\frac1{2p}}|y|^{-\frac{N-1}{2}+\frac{N-1}{p}},
& |y|\ge \sqrt t .
\end{cases}
\end{equation}
We split
\[
u_{\rm large}=u_{\rm large}^{(1)}+u_{\rm large}^{(2)},
\]
where \(u_{\rm large}^{(1)}\) corresponds to \(|y|\le\sqrt t\), and
\(u_{\rm large}^{(2)}\) corresponds to \(|y|\ge\sqrt t\).

First consider the region \(|y|\le \sqrt t\). Choose
\[
M=\nu-\mu+2n+2+\frac{N-1}{2}.
\]
Since
\[
\nu-\mu+2n+2+\frac{N-1}{2}
=
\nu+2n+\frac52>0,
\]
this is admissible in Lemma~\ref{localization}. Moreover,
\[
-\frac{N-1}{2}+M
=
\nu-\mu+2n+2.
\]
Using \eqref{eq11} and \eqref{idk1}, we obtain
\[
\begin{aligned}
\|u_{\rm large}^{(1)}(t)\|_{L^p(\R^N)}
&\le
C t^{-1/2}
t^{-\frac{N-1}{4}+\frac N{2p}-\frac M2}
\int_{\{|y|\le\sqrt t\}}
|y|^{-\frac{N-1}{2}+M}|u_0(y)|\,dy
\\
&=
C t^{-1/2-\frac{N-1}{4}+\frac N{2p}-\frac M2}
M_{\nu,n+1}(|u_0|) .
\end{aligned}
\]
By the moment assumption \eqref{momentassumption}, we get
\[
\|u_{\rm large}^{(1)}(t)\|_{L^p(\R^N)}
\le
C t^{-1/2-\frac{N-1}{4}+\frac N{2p}
-\frac12\left(\nu-\mu+2n+2+\frac{N-1}{2}\right)}.
\]
Using \(\mu=(N-2)/2\), the exponent simplifies to
\[
-\gamma_p-n-1,
\]
where $\gamma_p$ is defined in \eqref{gammapdef}.
Hence
\begin{equation}\label{ularge1}
\|u_{\rm large}^{(1)}(t)\|_{L^p(\R^N)}
\le
C M_{\nu,n+1}(|u_0|)\cdot t^{-\gamma_p-n-1}.
\end{equation}

Now consider the region \(|y|\ge\sqrt t\). From \eqref{idk1}, we have
\[
\left\|
|x|^{-\frac{N-1}{2}}
e^{-\frac{(|x|-|y|)^2}{4t}}
\mathbf 1_{\{|x||y|\ge 2t\}}
\right\|_{L^p_x}
\le
C t^{\frac1{2p}}
|y|^{-\frac{N-1}{2}+\frac{N-1}{p}}.
\]
Therefore
\[
\begin{aligned}
\|u_{\rm large}^{(2)}(t)\|_{L^p(\R^N)}
&\le
C t^{-\frac12+\frac1{2p}}
\int_{\{|y|\ge\sqrt t\}}
|y|^{-(N-1)+\frac{N-1}{p}}
|u_0(|y|)|\,dy .
\end{aligned}
\]

Since we assumed that $p\ge
\frac{2(N-1)}{N+2\nu+4n+4}$, it follows that
\[
-(N-1)+\frac{N-1}{p}
\le
\nu-\mu+2n+2.
\]
Then, on the set \(\{|y|\ge\sqrt t\}\), we have
\[
|y|^{-(N-1)+\frac{N-1}{p}}
\le
t^{\frac12\left(-(N-1)+\frac{N-1}{p}
-(\nu-\mu+2n+2)\right)}
|y|^{\nu-\mu+2n+2}.
\]
Hence
\[
\begin{aligned}
\|u_{\rm large}^{(2)}(t)\|_{L^p(\R^N)}
&\le
C t^{-\frac12+\frac1{2p}}
\cdot t^{\frac12\left(-(N-1)+\frac{N-1}{p}
-(\nu-\mu+2n+2)\right)}
\\
&\qquad \times
\int_{\{|y|\ge\sqrt t\}}
|y|^{\nu-\mu+2n+2}|u_0(|y|)|\,dy .
\end{aligned}
\]
Using again the moment assumption \eqref{momentassumption}, we obtain
\[
\|u_{\rm large}^{(2)}(t)\|_{L^p(\R^N)}
\le
C\cdot M_{\nu,n+1}(|u_0|)\cdot t^{-\frac12+\frac1{2p}}
t^{\frac12\left(-(N-1)+\frac{N-1}{p}
-(\nu-\mu+2n+2)\right)}.
\]
The exponent simplifies to
\[
-\gamma_p-n-1.
\]
Therefore
\begin{equation}\label{ularge2}
\|u_{\rm large}^{(2)}(t)\|_{L^p(\R^N)}
\le
C\cdot M_{\nu,n+1}(|u_0|)\cdot t^{-\gamma_p-n-1}.
\end{equation}

Combining \eqref{ularge1} and \eqref{ularge2}, and using
\(M_{\nu,n+1}(|u_0|)\le \|u_0\|_{1,\nu,n+1}\), we obtain
\[\|u_{\rm large}(t)\|_{L^p(\R^N)}
\le
C\cdot M_{\nu,n+1}(|u_0|)\cdot t^{-\gamma_p-n-1},\quad t\geq 1\]
\textbf{Step 5: Completing the proof}

We arrive at the following decomposition:
\[
u(x,t)=
A_n(x,t)
+
\mathcal E_{I,n}(x,t)
+
\mathcal E_{G,n}(x,t)
+
\mathcal E_{\chi,n}(x,t)+u_{\rm large}(x,t)
\]
with $A_n(x,t)$ defined as
\begin{equation}
   A_n(x,t):=
\sum_{m=0}^{n}
\sum_{\ell=0}^{n-m}
a_{m,\nu}
\frac{(-1)^\ell}{4^\ell\ell!}
t^{-1-\nu-2m-\ell}
|x|^{\nu-\mu+2m}
e^{-\frac{|x|^2}{4t}}
M_{\nu,m+\ell}(u_0). 
\end{equation}

We now claim that $A_n(x,t)$ coincides with the profile defined in \eqref{asymptoticprofiledef}. Indeed, set $j:=m+\ell. $ For a fixed \(j\), all pairs \((m,\ell)\) with \(m+\ell=j\) contribute to the
same moment \(M_{\nu,j}(u_0)\). Since \(\ell=j-m\), the contribution of total
order \(j\) equals
\[
M_{\nu,j}(u_0)
\sum_{m=0}^{j}
a_{m,\nu}
\frac{(-1)^{j-m}}{4^{j-m}(j-m)!}
t^{-1-\nu-2m-(j-m)}
|x|^{\nu-\mu+2m}
e^{-\frac{|x|^2}{4t}}.
\]
Using
\[
t^{-1-\nu-2m-(j-m)}
|x|^{\nu-\mu+2m}
=
t^{-1-\nu-j}
|x|^{\nu-\mu}
\left(\frac{|x|^2}{t}\right)^m,
\]
we obtain
\[
A_n(x,t)
=
\sum_{j=0}^{n}
M_{\nu,j}(u_0)
t^{-1-\nu-j}
|x|^{\nu-\mu}
e^{-\frac{|x|^2}{4t}}
P_j\left(\frac{|x|^2}{t}\right),
\]
where
\[
P_j(s)
:=
\sum_{m=0}^{j}
a_{m,\nu}
\frac{(-1)^{j-m}}{4^{j-m}(j-m)!}
s^m.
\]
Equivalently,
\[P_j(s):=\sum_{m=0}^{j}\frac{(-1)^{j-m}}{2^{2(\nu+j+m)+1}m!(j-m)!\Gamma(m+\nu+1)}
s^m\]
which is precisely the polynomial defined in \eqref{pndef}. With this notation,
\[ A_n(x,t)
=
\sum_{j=0}^{n}\mathcal{P}_j(x,t).
\]
and
\[
u(x,t)
=
A_n(x,t)
+
\mathcal E_{I,n}(x,t)
+
\mathcal E_{G,n}(x,t)
+
\mathcal E_{\chi,n}(x,t)
+
u_{\rm large}(x,t).
\]
and the claim is proven. We now take the \(L^p(\mathbb R^N)\)-norm. By the triangle inequality,
\[
\left\|
u(t)-A_n(t)
\right\|_{L^p(\R^N)}
\le
\|\mathcal E_{I,n}(t)\|_{L^p(\R^N)}
+
\|\mathcal E_{G,n}(t)\|_{L^p(\R^N)}
+
\|\mathcal E_{\chi,n}(t)\|_{L^p(\R^N)}
+
\|u_{\rm large}(t)\|_{L^p(\R^N)}.
\]

From the estimates obtained in the previous steps, we have
\[
\left\|
u(t)-A_n(t)
\right\|_{L^p(\R^N)}
\le
C\cdot M_{\nu,n+1}(|u_0|)\cdot t^{-\gamma_p-n-1}
\]
for all \(t\ge1\). This completes the proof.
\end{proof}

\begin{corollary}\label{vanishingmomentestimate}
Let \(f:\mathbb R^N\to\mathbb R\) be radial. Assume
\[
\int_{\mathbb R^N}
|x|^{\nu-\mu}\bigl(1+|x|^{2n}\bigr)|f(x)|\,dx
<\infty.
\]
and
\[
M_{\nu,j}(f)=0,
\qquad j=0,\dots,n-1.
\]
Let $p\in[\frac{2(N-1)}{N+2\nu+4n},\frac{N}{\mu-\nu})$. Then there exists \(C>0\), independent of \(f\) and \(t\), such that for all
\(t\ge1\),
\[
\|S_\lambda(t)f\|_{L^p(\mathbb R^N)}
\le
C t^{-\gamma_p-n}
M_{\nu,n}(|f|).
\]
where
\[S_{\lambda}(t)f=\int_{\R^N}K(x,y,t)f(y)dy.\]
\end{corollary}

\begin{proof}
By Proposition \ref{asymptoticbound} applied for $n-1$ instead of $n$, we have that
\[
\left\|
S_\lambda(t)f-A_{n-1}^f(t)
\right\|_{L^p(\R^N)}
\le
C t^{-\gamma_p-n}M_{\nu,n}(|f|).
\]
But by assumption,
\[
M_{\nu,j}(f)=0,
\qquad j=0,\dots,n-1.
\]
so $
A_{n-1}^f\equiv0, $
Therefore
\[
\|S_\lambda(t)f\|_{L^p(\R^N)}
\le
C t^{-\gamma_p-n}M_{\nu,n}(|f|).
\]
\end{proof}

We now have all the ingredients to prove the first main result:
\begin{proof}[\textbf{Proof of Theorem~\ref{asymptoticprofilethm}}]
We argue by density. By Lemma \ref{bumpfunctions}, there exist radial functions $
\eta_0,\eta_1,\dots,\eta_n\in C_c^\infty(\mathbb R^N) $
such that
\[
M_{\nu,i}(\eta_j)=\delta_{ij},
\qquad
0\le i,j\le n.
\]
Let \(\chi\in C_c^\infty([0,\infty))\) satisfy
\[
0\le\chi\le1,
\qquad
\chi(s)=1\text{ for }0\le s\le1,
\qquad
\chi(s)=0\text{ for }s\ge2.
\]
For \(k\in\mathbb N\), define
\[
\psi_k(x)
:=
\chi\left(\frac{|x|}{k}\right)u_0(|x|).
\]
Let 
\[
\|f\|_{L^1_{\nu,n}}
:=
\int_{\mathbb R^N}
|x|^{\nu-\mu}(1+|x|^{2n})|f(x)|\,dx.
\]
Since $\|u_0\|_{L^1_{\nu,n}}<\infty$, it follows by the Dominated Convergence Theorem that 
\begin{equation}\label{eq16}
    \psi_k\to u_0 \quad \text{in}\quad L^1_{\nu,n}(\mathbb R^N). 
\end{equation}
Therefore for all $j=0,\dots n$, we have
\begin{equation}\label{eq17}
    \begin{split}
        |M_{\nu,j}(\psi_k)-M_{\nu,j}(u_0)|&\leq\int_{\R^N}|x|^{\nu-\mu+2j}\left|\psi_k(x)-u_0(x)\right| dx\\
        &\leq \int_{\R^N}|x|^{\nu-\mu}\left(1+|x|^{2n}\right)\left|\psi_k(x)-u_0(x)\right|dx\\
        &=\|\psi_k-u_0\|_{L^1_{\nu,n}}\stackrel{k\to\infty}{\longrightarrow}0.
    \end{split}
\end{equation}
Thus $M_{\nu,j}(\psi_k)\to M_{\nu,j}(u_0)$ for all $j=0,1,\dots,n$.
Now set
\[
\phi_k := \psi_k + \sum_{j=0}^{n}
\bigl(M_{\nu,j}(u_0)-M_{\nu,j}(\psi_k)\bigr)\eta_j.
\]
Then, for every \(i=0,\dots,n\),
\begin{equation}\label{eq18}
    \begin{split}
        M_{\nu,i}(\phi_k)&=M_{\nu,i}(\psi_k)+\sum_{j=0}^n\left(M_{\nu,j}(u_0)-M_{\nu,j}(\psi_k)\right)M_{\nu,i}(\eta_j)\\
        &=M_{\nu,i}(\psi_k)+\sum_{j=0}^n\left(M_{\nu,j}(u_0)-M_{\nu,j}(\psi_k)\right)\delta_{ij}\\
        &=M_{\nu,i}(\psi_k)+M_{\nu,i}(u_0)-M_{\nu,i}(\psi_k)\\
        &=M_{\nu,i}(u_0).
    \end{split}
\end{equation}
It follows that the asymptotic
profiles of $\phi_k$ and $u_0$ agree:
\[
A_n^{\phi_k}(x,t)=A_n^{u_0}(x,t),
\]
Thus
\begin{equation}\label{profilesplit}
    S_\lambda(t)u_0 - A_n^{u_0}(t)
=
S_\lambda(t)(u_0-\phi_k)
+
\left(
S_\lambda(t)\phi_k
-
A_n^{\phi_k}(t)
\right).
\end{equation}
Multiplying by \(t^{\gamma_p+n}\) and taking the \(L^p\)-norm gives
\[
\begin{aligned}
&t^{\gamma_p+n}
\left\|
S_\lambda(t)u_0-A_n^{u_0}(t)
\right\|_{L^p(\R^N)}
\\
&\qquad\le
t^{\gamma_p+n}
\|S_\lambda(t)(u_0-\phi_k)\|_{L^p(\R^N)}
\\
&\qquad\qquad+
t^{\gamma_p+n}
\left\|
S_\lambda(t)\phi_k
- A_n^{\phi_k}(t)
\right\|_{L^p(\R^N)}.
\end{aligned}
\]
We estimate the two terms separately. By \eqref{eq18}, we have $
M_{\nu,j}(u_0-\phi_k)=0,$ for all $j=0,\dots,n, $ thus Corollary \ref{vanishingmomentestimate}  gives
\[
t^{\gamma_p+n}
\|S_\lambda(t)(u_0-\phi_k)\|_{L^p(\R^N)}
\le
C
M_{\nu,n}(|u_0-\phi_k|).
\]
We now estimate the second term. Since each \(\eta_j\in C_c^\infty(\mathbb R^N)\subset L^1_{\nu,n}(\mathbb R^N)\), and since the coefficients
\[
M_{\nu,j}(u_0)-M_{\nu,j}(\psi_k)
\]
tend to \(0\) by \eqref{eq17}, it follows from \eqref{eq16} that
\[
\phi_k\to u_0
\quad\text{in }L^1_{\nu,n}(\mathbb R^N).
\]
\[1/2\]
In particular, $\|\phi_k\|_{L^1_{\nu,n}}<\infty$ so 
\begin{equation}\label{nustiu3}
    \int_{\R^N}|x|^{\nu-\mu}|\phi_k(x)|dx<\infty,\qquad\int_{\R^N}|x|^{\nu-\mu+2n}|\phi_k(x)|dx<\infty
\end{equation} Since \(\phi_k\) is supported in some ball \(\{|x|\le r_k\}\), we have on this set
\[
|x|^{\nu-\mu+2n+2}
=
|x|^2 |x|^{\nu-\mu+2n}
\le r_k^2 |x|^{\nu-\mu+2n}.
\]
Therefore
\begin{equation}\label{eq19}
    \int_{\mathbb R^N}
|x|^{\nu-\mu+2n+2}|\phi_k(x)|\,dx
\le
r_k^2 M_{\nu,n}(|\phi_k|)
<\infty.
\end{equation}
Combining \eqref{nustiu3} and \eqref{eq19} we get that $\phi_k$ satisfies the stronger moment condition 
\[\|\phi_k\|_{L^1_{\nu,n+1}}<\infty\] required in Proposition \ref{asymptoticbound}. Applying this result, it follows that
\[
\left\|
S_\lambda(t)\phi_k
- A_n^{\phi_k}(t)
\right\|_{L^p(\R^N)}
\le
C\cdot M_{\nu,n+1}(|\phi_k|)\cdot t^{-\gamma_p-n-1}
\]
for $t\geq 1$. Therefore, for fixed \(k\),
\[
t^{\gamma_p+n}
\left\|
S_\lambda(t)\phi_k
-A_n^{\phi_k}(t)
\right\|_{L^p(\R^N)}
\le
C\cdot M_{\nu,n+1}(|\phi_k|)\cdot t^{-1}\to0
\qquad\text{as }t\to\infty.
\]

Taking the \(\limsup\) of \eqref{profilesplit} as \(t\to\infty\), we get
\[
\limsup_{t\to\infty}
t^{\gamma_p+n}
\left\|
S_\lambda(t)u_0-A_n^{u_0}(t)
\right\|_{L^p(\R^N)}
\le
C \cdot M_{\nu,n}(|u_0-\phi_k|)\leq C
\|u_0-\phi_k\|_{L^1_{\nu,n}}.
\]
Finally, letting \(k\to\infty\), and using
\[
\phi_k\to u_0
\qquad\text{in }L^1_{\nu,n}(\mathbb R^N),
\]
we obtain
\[
\limsup_{t\to\infty}
t^{\gamma_p+n}
\left\|
S_\lambda(t)u_0-A_n^{u_0}(t)
\right\|_{L^p(\R^N)}
=0.
\]
Hence
\[
\lim_{t\to\infty}
t^{\gamma_p+n}
\left\|
S_\lambda(t)u_0-A_n^{u_0}(t)
\right\|_{L^p(\R^N)}
=0.
\]
This proves the theorem.
\end{proof}

\begin{proof}[\textbf{Proof of Remark \ref{sharpness}} (Sharpness)]
Assume that
\[
M_{\nu,0}(u_0)=\cdots=M_{\nu,n-1}(u_0)=0,
\qquad
M_{\nu,n}(u_0)\neq0.
\]
Note that there exist such functions due to Lemma \ref{bumpfunctions}. Then \(A_n=\mathcal P_n\). By Theorem~\ref{asymptoticprofilethm},
\[
S_\lambda(t)u_0-\mathcal P_n(\cdot,t)
=
o\left(t^{-\gamma_p-n}\right)
\quad\text{in }L^p(\mathbb R^N).
\]
On the other hand, Lemma~\ref{lemma0.2} gives
\[
\|\mathcal P_n(\cdot,t)\|_{L^p(\mathbb R^N)}
=
C_{n,p,\nu}|M_{\nu,n}(u_0)|t^{-\gamma_p-n}.
\]
Therefore, by the $||a|-|b||\leq |a-b|$ inequality,
\[
\|S_\lambda(t)u_0\|_{L^p(\mathbb R^N)}
=
C_{n,p,\nu}|M_{\nu,n}(u_0)|t^{-\gamma_p-n}
+
o\left(t^{-\gamma_p-n}\right).
\]
Hence the decay rate \(t^{-\gamma_p-n}\) cannot, in general, be improved.
\end{proof}

The non radial asymptotic results now follow from Theorem \ref{asymptoticprofilethm}:
\begin{proof}[\textbf{Proof of Theorem \ref{nonradial-finite}}]
By the non-radial Hardy heat kernel formula obtained in
Proposition~\ref{nonradialkernel}, the solution corresponding to the finite angular
expansion
\[
u_0(r,\omega)
=
\sum_{\ell=0}^{L}
\sum_{m=1}^{d_\ell}
u_{0,\ell,m}(r)Y_{\ell,m}(\omega)
\]
is given by
\[
u(x,t)=u(r,\omega,t)
=
\sum_{\ell=0}^{L}
\sum_{m=1}^{d_\ell}
u_{\ell,m}(r,t)Y_{\ell,m}(\omega),
\]
where
\[
u_{\ell,m}(r,t)
=
\int_0^\infty
K_{\ell}(r,\rho,t)
u_{0,\ell,m}(\rho)\rho^{N-1}\,d\rho,
\]
and
\[
K_{\ell}(r,\rho,t)
=
\frac{1}{2t}
(r\rho)^{-\mu}
\exp\!\left(-\frac{r^2+\rho^2}{4t}\right)
I_{\nu_\ell}\left(\frac{r\rho}{2t}\right).\qquad \nu_\ell=\sqrt{(\ell+\mu)^2-\lambda}.
\]
Applying Theorem~\ref{asymptoticprofilethm} to each
\(u_{0,\ell,m}\), with \(\nu\) replaced by \(\nu_\ell\), gives
\[
\lim_{t\to\infty}
t^{\gamma_{p,\ell}+n}
\left\|
u_{\ell,m}(\cdot,t)
-
\sum_{j=0}^{n}\mathcal P_{\ell,m,j}(\cdot,t)
\right\|_{L^p(\R^N)}
=0,
\]
where $\gamma_{p,\ell}$ is defined in \eqref{gammapldef}. Since \(\nu_\ell\geq\nu_0=\nu\), we have $
\gamma_{p,\ell}\geq\gamma_p. $
Therefore
\[
\lim_{t\to\infty}
t^{\gamma_p+n}
\left\|
u_{\ell,m}(\cdot,t)
-
\sum_{j=0}^{n}\mathcal P_{\ell,m,j}(\cdot,t)
\right\|_{L^p(\R^N)}
=0.
\]

Since the angular expansion is finite and each spherical harmonic
\(Y_{\ell,m}\) is bounded on \(\mathbb S^{N-1}\), we may sum over
\(\ell=0,\dots,L\) and \(m=1,\dots,d_\ell\). Hence
\[
\begin{aligned}
& t^{\gamma_p+n}
\left\|
u(\cdot,t)-A_{n,L}(\cdot,t)
\right\|_{L^p(\mathbb R^N)}
\\
&\qquad\le
C
\sum_{\ell=0}^{L}
\sum_{m=1}^{d_\ell}
t^{\gamma_p+n}
\left\|
u_{\ell,m}(\cdot,t)
-
\sum_{j=0}^{n}\mathcal P_{\ell,m,j}(\cdot,t)
\right\|_{L^p(\R^N)}.
\end{aligned}
\]
Each term on the right-hand side tends to \(0\), and the sum is finite. Therefore
\[
\lim_{t\to\infty}`
t^{\gamma_p+n}
\left\|
u(\cdot,t)-A_{n,L}(\cdot,t)
\right\|_{L^p(\mathbb R^N)}
=0.
\]
\end{proof}

\begin{proof}[\textbf{Proof of Theorem \ref{nonradial-infinite}}]
For each \((\ell,m)\), define the radial remainder
\[
R_{\ell,m}(r,t)
:=
u_{\ell,m}(r,t)
-
\sum_{j=0}^{n}\mathcal P_{\ell,m,j}(r,t).
\]
Then, by the definition of \(A_n^\infty\), we have
\[
u(r\omega,t)-A_n^\infty(r\omega,t)
=
\sum_{\ell=0}^{\infty}
\sum_{m=1}^{d_\ell}
R_{\ell,m}(r,t)Y_{\ell,m}(\omega).
\]

We apply Theorem \ref{asymptoticprofilethm} to each angular mode, with \(\nu\) replaced by \(\nu_\ell\). This gives, for every fixed \((\ell,m)\),
\[
\lim_{t\to\infty}
t^{\gamma_{p,\ell}+n}
\|R_{\ell,m}(\cdot,t)\|_{L^p(\R^N)}
=0.
\]
Since \(\nu_\ell\geq \nu\), we have \(\gamma_{p,\ell}\geq \gamma_p\). Hence
\[
\lim_{t\to\infty}
t^{\gamma_p+n}
\|R_{\ell,m}(\cdot,t)\|_{L^p(\R^N)}
=0.
\]

Moreover, Proposition~\ref{asymptoticbound}, applied with parameter
\(\nu_\ell\), yields a constant \(C_{\ell,n,p}>0\), independent of \(t\), such
that, for every \(t\geq1\),
\[
t^{\gamma_p+n}
\|R_{\ell,m}(\cdot,t)\|_{L^p(\R^N)}
\leq
C_{\ell,n,p}
\|u_{0,\ell,m}\|_{L^1_{\nu_\ell,n+1}}.
\]

We now estimate the full remainder. By Minkowski's inequality and the separation
of variables in polar coordinates,
\[
\begin{aligned}
&t^{\gamma_p+n}
\left\|
u(\cdot,t)-A_n^\infty(\cdot,t)
\right\|_{L^p(\mathbb R^N)}
\\
&\qquad\leq
\sum_{\ell=0}^{\infty}
\sum_{m=1}^{d_\ell}
t^{\gamma_p+n}
\left\|
R_{\ell,m}(r,t)Y_{\ell,m}(\omega)
\right\|_{L^p(\mathbb R^N)}
\\
&\qquad=
\sum_{\ell=0}^{\infty}
\sum_{m=1}^{d_\ell}
\|Y_{\ell,m}\|_{L^p(\mathbb S^{N-1})}
t^{\gamma_p+n}
\|R_{\ell,m}(\cdot,t)\|_{L^p(\R^N)}
\\
&\qquad\leq
\sum_{\ell=0}^{\infty}
\sum_{m=1}^{d_\ell}
\|Y_{\ell,m}\|_{L^p(\mathbb S^{N-1})}
C_{\ell,n,p}
\|u_{0,\ell,m}\|_{L^1_{\nu_\ell,n+1}}.
\end{aligned}
\]
The last series is finite by assumption.

However, for every fixed \((\ell,m)\),
\[
\|Y_{\ell,m}\|_{L^p(\mathbb S^{N-1})}
t^{\gamma_p+n}
\|R_{\ell,m}(\cdot,t)\|_{L^p(\R^N)}
\longrightarrow 0
\qquad\text{as }t\to\infty.
\]
Therefore, by the Dominated Convergence Theorem for series,
\[
\sum_{\ell=0}^{\infty}
\sum_{m=1}^{d_\ell}
\|Y_{\ell,m}\|_{L^p(\mathbb S^{N-1})}
t^{\gamma_p+n}
\|R_{\ell,m}(\cdot,t)\|_{L^p(\R^N)}
\longrightarrow 0.
\]
Consequently,
\[
\lim_{t\to\infty}
t^{\gamma_p+n}
\left\|
u(\cdot,t)-A_n^\infty(\cdot,t)
\right\|_{L^p(\mathbb R^N)}
=0.
\]
This proves the theorem.
\end{proof}

\section{Appendix}
The derivation of the kernel for problem \eqref{problem} makes use of some folowing notions, such as Bessel functions and Hankel transform, for which we reffer the reader to \cite{W},\cite{CCFI} and \cite{EK} respectively.
\begin{prop}[Radial heat kernel]\label{radialkernel}
Let \(u_0\in C_c^\infty(\mathbb R^N\setminus\{0\})\) be radial. The function \(K(x,y,t)\), defined for \(x,y\in\mathbb R^N\setminus\{0\}\), \(t>0\) by
\[
K(x,y,t)
:=
\frac{1}{|\mathbb{S}^{N-1}|}\frac{1}{2t}\,(|x||y|)^{-\frac{N-2}{2}}
\exp\!\left(-\frac{|x|^2+|y|^2}{4t}\right)\,
I_\nu\!\left(\frac{|x||y|}{2t}\right),
\]
generates a heat kernel for \eqref{problem}. That is, the radial semigroup
\[
u(x,t)
=
\int_{\mathbb R^N} K(x,y,t)\,u_0(y)\,dy
\]
belongs to \(C^\infty( (\mathbb R^N\setminus\{0\})\times(0,\infty))\), and pointwise satisfies $u_t=\Delta u+\lambda|x|^{-2}u$. Moreover, \(K(x,y,t)>0\), and \(K\) is fundamental in the radial sense:
for every radial \(u_0\in C_c^\infty(\mathbb R^N\setminus\{0\})\),
\[
\lim_{t\downarrow0}\int_{\mathbb R^N}K(x,y,t)u_0(y)\,dy=u_0(x),
\qquad x\neq0.
\]
\end{prop}
\begin{proof} 
If we assume $u$ is radial, i.e.\ $u(x,t)=u(r,t)$ with $r=|x|$, then \eqref{problem} becomes
\begin{equation}\label{eq:radial-heat}
\partial_t u
= u_{rr} + \frac{N-1}{r}u_r + \frac{\lambda}{r^2}u.
\end{equation}

It is convenient to remove the first derivative by the standard conjugation $u(r,t) = r^{-\mu}\,v(r,t)$. A short computation gives
\[
u_r = r^{-\mu}\Big(v_r-\frac{\mu}{r}v\Big),\qquad
u_{rr} = r^{-\mu}\Big(v_{rr}-\frac{2\mu}{r}v_r+\frac{\mu(\mu+1)}{r^2}v\Big),
\]
hence
\[
u_{rr}+\frac{N-1}{r}u_r
= r^{-\mu}\left(
v_{rr}+\frac{1}{r}v_r - \frac{\mu^2}{r^2}v
\right),
\]
because $N-1-2\mu = 1$ and $\mu(\mu+1)-(N-1)\mu=-\mu^2$.
Plugging into \eqref{eq:radial-heat} yields the Bessel-type equation
for $v$:
\begin{equation}\label{eq:v-eq}
\partial_t v
= v_{rr}+\frac{1}{r}v_r - \frac{\nu^2}{r^2}\,v.
\end{equation}

In the Hardy range $\lambda<\lambda_*$, we have $\nu>0$ (and $\nu=0$ for $\lambda=\lambda_*$).
Let $\mathcal{L}_\nu := \partial_{rr}+\frac{1}{r}\partial_r - \frac{\nu^2}{r^2}. $
Seeking separated solutions $v(r,t)=\ee^{-k^2 t}\phi_k(r)$ turns \eqref{eq:v-eq}
into the eigenvalue equation
\begin{equation}\label{eq:eig}
\phi_k''(r)+\frac{1}{r}\phi_k'(r) + \left(k^2-\frac{\nu^2}{r^2}\right)\phi_k(r)=0.
\end{equation}
The regular solution at $r=0$ is (up to a constant)
\[
\phi_k(r)=J_\nu(kr),
\]
where $J_\nu$ is the Bessel function \cite{W} of the first kind.

We now introduce the Hankel transform \cite{SN} of order $\nu$ :
\[
(\mathcal{H}_\nu f)(k) \coloneqq \int_0^\infty f(r)\,J_\nu(kr)\,r\,\dd r,
\qquad
f(r)=\int_0^\infty (\mathcal{H}_\nu f)(k)\,J_\nu(kr)\,k\,\dd k,
\]

Applying the Hankel transform to $\mathcal L_\nu f$ for $f\in C^{\infty}_c(\R^N\setminus\{0\})$ gives
\[
\mathcal H_\nu(\mathcal L_\nu f)(k)
=
\int_0^\infty
\left(
f''(r)+\frac1r f'(r)-\frac{\nu^2}{r^2}f(r)
\right)
J_\nu(kr)\,r\,\dd r .
\]
We treat the terms separately. First,
\[
\int_0^\infty f''(r)J_\nu(kr)\,r\,\dd r
=
-\int_0^\infty f'(r)\frac{\dd}{\dd r}\!\big(rJ_\nu(kr)\big)\,\dd r,
\]
All boundary terms vanish because \(f\in C_c^\infty(0,\infty)\); in particular,
\(f\) and all its derivatives vanish near \(0\) and near \(+\infty\). Since
\[
\frac{\dd}{\dd r}\big(rJ_\nu(kr)\big)
=
J_\nu(kr)+krJ_\nu'(kr),
\]
we obtain
\[
\int_0^\infty f''(r)J_\nu(kr)\,r\,\dd r
=
-\int_0^\infty f'(r)J_\nu(kr)\,\dd r
-
k\int_0^\infty f'(r)rJ_\nu'(kr)\,\dd r .
\]
Substituting above, we obtain
\[
\mathcal H_\nu(\mathcal L_\nu f)(k)
=
-k\int_0^\infty f'(r)rJ_\nu'(kr)\,\dd r
-\int_0^\infty \frac{\nu^2}{r}f(r)J_\nu(kr)\,\dd r .
\]
Integrating the first term by parts once more yields
\[
-k\int_0^\infty f'(r)rJ_\nu'(kr)\,\dd r
=
k\int_0^\infty f(r)\frac{\dd}{\dd r}\!\big(rJ_\nu'(kr)\big)\,\dd r .
\]
Using the Bessel equation
\[
(kr)^2J_\nu''(kr)+(kr)J_\nu'(kr)+(k^2r^2-\nu^2)J_\nu(kr)=0,
\]
one finds
\[
k\frac{\dd}{\dd r}\big(rJ_\nu'(kr)\big)
=
-k^2rJ_\nu(kr)+\frac{\nu^2}{r}J_\nu(kr).
\]
Substituting this identity gives
\[
\mathcal H_\nu(\mathcal L_\nu f)(k)
=
-k^2\int_0^\infty f(r)J_\nu(kr)\,r\,\dd r
=
-k^2(\mathcal H_\nu f)(k).
\]
Applying this to \eqref{eq:v-eq} gives
\[
\partial_t \widehat{v}(t,k) = -k^2 \widehat{v}(t,k),
\qquad \widehat{v}\coloneqq \mathcal{H}_\nu v,
\]
from which we obtain $\hat{v}(t,k)=\hat{v}(k,0)e^{-k^2t}$. Using the inverse Hankel transform we obtain
\begin{equation}\label{nustiu2}
    v(r,t)=\int_0^\infty \widehat{v}(k,t)J_\nu(kr)\,k\,\dd k=\int_0^\infty e^{-k^2 t}\widehat v(k,0)J_\nu(kr)\,k\,\dd k.
\end{equation}
By definition of the Hankel transform,
\[
\widehat v(k,0)=\int_0^\infty v_0(\rho)J_\nu(k\rho)\,\rho\,\dd\rho,
\]
Substituting this into the previous formula yields
\[
v(r,t)
=
\int_0^\infty e^{-k^2 t}
\left(
\int_0^\infty v_0(\rho)J_\nu(k\rho)\rho\,\dd\rho
\right)
J_\nu(kr)k\,\dd k .
\]
Interchanging the order of integration we obtain
\begin{equation}
   \label{nustiu1}
v(r,t)
=
\int_0^\infty
\left(
\int_0^\infty e^{-k^2 t}J_\nu(kr)J_\nu(k\rho)k\,\dd k
\right)
v_0(\rho)\rho\,\dd\rho . 
\end{equation}
A known result (see e.g. \cite{GZ}[Eq.~6.633.2]) gives
\begin{equation}\label{eq:weber}
\int_0^\infty \ee^{-t k^2}J_\nu(kr)J_\nu(k\rho)\,k\,\dd k
=
\frac{1}{2t}\exp\!\left(-\frac{r^2+\rho^2}{4t}\right)
I_\nu\!\left(\frac{r\rho}{2t}\right),
\end{equation}
where $I_\nu$ is the modified Bessel function of the first kind. Recall $u(r,t)=r^{-\mu}v(r,t)$, thus $v_0(\rho)=\rho^{\mu} u_0(\rho)$. Coming back to spatial coordinates, we have obtained
\[u(x,t)=\int_{\R^N}K(x,y,t)u_0(y) dy\]
with
\begin{equation}\label{eq:final}
K(x,y,t)
=
\frac{1}{|\mathbb{S}^{N-1}|}\frac{1}{2t}\,(|x||y|)^{-\frac{N-2}{2}}
\exp\!\left(-\frac{|x|^2+|y|^2}{4t}\right)\,
I_{\nu}\!\left(\frac{|x||y|}{2t}\right),
\end{equation}
Since \(r,\rho,t>0\) and \(I_\nu(z)>0\) for \(z>0\), \(\nu\ge0\), all factors in \(K\) are positive. We now verify that the mild solution $u$ indeed satisfies the Cauchy problem \eqref{problem}. From the representation derived in \eqref{nustiu1}
\[
v(r,t)
=
\int_0^\infty
\left(
\int_0^\infty e^{-k^2 t}J_\nu(kr)J_\nu(k\rho)k\,\dd k
\right)
v_0(\rho)\rho\,\dd\rho .
\]
we observe that for every $t>0$ the Gaussian factor $e^{-k^2 t}$ ensures rapid decay in $k$.
Therefore we may differentiate under the integral sign and apply the operator
$\mathcal L_\nu$ in the $r$ variable and, using that
\[
\mathcal L_\nu[J_\nu(kr)] = -k^2 J_\nu(kr),
\]
we obtain
\[
\partial_t K_\nu(t;r,\rho)
=
\int_0^\infty (-k^2)e^{-k^2 t}J_\nu(kr)J_\nu(k\rho)\,k\,\dd k
=
\mathcal (\mathcal{L}_\nu)_r (K_\nu(t;r,\rho)).
\]
Consequently,
\[
\partial_t v(r,t)
=
\int_0^\infty \partial_t K_\nu(t;r,\rho)\,v_0(\rho)\,\rho\,\dd\rho
=\int_0^\infty (\mathcal{L}_{\nu})_r( K_\nu(t;r,\rho))\,v_0(\rho)\,\rho\,\dd\rho
=
\mathcal L_\nu v(r,t),
\]
that is,
\[
\partial_t v
=
v_{rr}+\frac{1}{r}v_r-\frac{\nu^2}{r^2}v.
\]
Recalling that $u=r^{-\mu}v$, reversing the computation performed earlier shows that
\[
\partial_t u
=
u_{rr}+\frac{N-1}{r}u_r+\frac{\lambda}{r^2}u,
\qquad t>0,\ r>0.
\]
 mIt remains to verify the initial condition. By \eqref{nustiu2} we have
\[
v(r,t)
=
\int_0^\infty e^{-k^2 t}\widehat v_0(k)J_\nu(kr)\,k\,\dd k.
\]
We first note that \(\widehat v_0\) is rapidly decaying. Indeed, since
\(v_0\in C_c^\infty(0,\infty)\), we have
\(\mathcal L_\nu^m v_0\in C_c^\infty(0,\infty)\) for every \(m\in\mathbb N\). Moreover, from
\[
\mathcal H_\nu(\mathcal L_\nu f)(k)
=
-k^2\mathcal H_\nu f(k),
\]
we obtain, by iteration with \(f=v_0\),
\begin{equation}\label{iteration}
    \mathcal H_\nu v_0(k)
=
\frac{(-1)^m}{k^{2m}}
\mathcal H_\nu(\mathcal L_\nu^m v_0)(k),
\qquad k>0.
\end{equation}

Since \(\mathcal L_\nu^m v_0\in C_c^\infty(0,\infty)\) and \(J_\nu\) is bounded on
\([0,\infty)\), the function
\[
\mathcal H_\nu(\mathcal L_\nu^m v_0)(k)
=
\int_0^\infty (\mathcal L_\nu^m v_0)(\rho)J_\nu(k\rho)\rho\,\dd\rho
\]
is bounded in \(k\). Hence, for \(k\ge1\), \eqref{iteration} gives
\[
|\widehat v_0(k)|
=
|\mathcal H_\nu v_0(k)|
\le C_m k^{-2m}.
\]
Since \(m\) is arbitrary, \(\widehat v_0\) decays faster than any negative power
of \(k\). Choosing \(m\) large enough, for instance \(m\ge2\), we get
\[
|\widehat v_0(k)|\,|J_\nu(kr)|\,k\in L^1(1,\infty).
\]
On the other hand, on \((0,1)\), the same function is integrable because
\(\widehat v_0\) and \(J_\nu\) are bounded and \(k\mapsto k\in L^1(0,1)\). Therefore
\[
k\mapsto |\widehat v_0(k)|\,|J_\nu(kr)|\,k
\]
belongs to \(L^1(0,\infty)\) for every fixed \(r>0\). We now pass to the limit \(t\downarrow0\). Since
\[
|e^{-k^2t}\widehat v_0(k)J_\nu(kr)k|
\le
|\widehat v_0(k)|\,|J_\nu(kr)|\,k,
\]
and \(e^{-k^2t}\to1\) pointwise as \(t\downarrow0\), the Dominated Convergence
Theorem gives
\[
v(r,t)
\longrightarrow
\int_0^\infty \widehat v_0(k)J_\nu(kr)\,k\,\dd k.
\]
By the inverse Hankel transform, the last integral equals \(v_0(r)\). Therefore
\[
v(r,t)\longrightarrow v_0(r),
\qquad \text{as } t\downarrow0.
\]

Recalling that \(u(r,t)=r^{-\mu}v(r,t)\), we obtain, for every \(r>0\),
\[
u(r,t)
=
r^{-\mu}v(r,t)
\longrightarrow
r^{-\mu}v_0(r)
=
u_0(r),
\qquad \text{as } t\downarrow0.
\]
Thus the initial condition is recovered pointwise away from the origin. In
particular, the convergence also holds in the sense of distributions on
\(\mathbb R^N\setminus\{0\}\).
This completes the proof.
\end{proof}

\begin{remark}
Since \(v_0\in C_c^\infty(0,\infty)\), all functions involved are supported away
from \(0\) and infinity. Hence all boundary terms in the integrations by parts
vanish. Moreover, for \(t>0\), the factor \(e^{-tk^2}\) gives enough decay in
\(k\) to justify Fubini's theorem and differentiation under the integral sign.
\end{remark}
In order to derive the non-radial hardy heat kernel, we make use of spherical harmonics and spherical harmonics decomposition of the laplacian. We refer the reader to \cite{ABR} for some standard facts on this topic.
\begin{prop}[Non-radial Hardy heat kernel]\label{nonradialkernel}
Let \(u_0\in C^\infty_c (\mathbb R^N\setminus\{0\})\). Let \(\{Y_{\ell,m}\}_{m=1}^{d_\ell}\) be an orthonormal basis of spherical
harmonics of degree \(\ell\) on \(\mathbb S^{N-1}\), satisfying
\[
    -\Delta_{\mathbb S^{N-1}}Y_{\ell,m}
    =
    \ell(\ell+N-2)Y_{\ell,m}.
\]
For each \(\ell\in\mathbb N_0\), define $
    \nu_\ell
    :=
    \sqrt{\mu^2+\ell(\ell+N-2)-\lambda}$
Then the heat kernel associated with problem \eqref{problem} is given, for \(x=r\omega\), \(y=\rho\eta\), \(r,\rho>0\), and
\(\omega,\eta\in\mathbb S^{N-1}\), by
\[
K(x,y,t)
=
\frac{1}{2t}(r\rho)^{-\mu}
\exp\!\left(-\frac{r^2+\rho^2}{4t}\right)
\sum_{\ell=0}^{\infty}
I_{\nu_\ell}\left(\frac{r\rho}{2t}\right)
\sum_{m=1}^{d_\ell}
Y_{\ell,m}(\omega)Y_{\ell,m}(\eta)
\]
In other words, the function
\[
u(x,t)
=
\int_{\mathbb R^N} K(x,y,t)\,u_0(y)\,dy
\]
belongs to \(C^\infty( (\mathbb R^N\setminus\{0\})\times(0,\infty))\) and satisfies \eqref{problem}
Moreover, \(K(x,y,t)\) is fundamental in the sense of distributions.
\end{prop}
\begin{remark}
    For $\ell=0$, we recover the radial heat kernel formula.
\end{remark}
\begin{proof}
Since the potential \(|x|^{-2}\) is radial, the angular variables can be separated.
In spherical coordinates \(x=r\omega\), one has the Laplacian formula:
\[
    \Delta
    =
    \partial_{rr}
    +
    \frac{N-1}{r}\partial_r
    +
    \frac{1}{r^2}\Delta_{\mathbb S^{N-1}}.
\]
We expand the initial datum and the solution in spherical harmonics:
\[
    u_0(r,\omega)
    =
    \sum_{\ell=0}^{\infty}
    \sum_{m=1}^{d_\ell}
    u_{0,\ell,m}(r)Y_{\ell,m}(\omega),
\]
where
\[
    u_{0,\ell,m}(r)
    =
    \int_{\mathbb S^{N-1}}
    u_0(r,\eta)Y_{\ell,m}(\eta)\,d\sigma(\eta),
\]
and similarly
\[
    u(r,\omega,t)
    =
    \sum_{\ell=0}^{\infty}
    \sum_{m=1}^{d_\ell}
    u_{\ell,m}(r,t)Y_{\ell,m}(\omega).
\]
Substituting this expansion into
\[
    u_t=\Delta u+\frac{\lambda}{r^2}u
\]
and using
\[
    \Delta_{\mathbb S^{N-1}}Y_{\ell,m}
    =
    -\ell(\ell+N-2)Y_{\ell,m},
\]
we obtain, for each pair \((\ell,m)\),
\[
    \partial_t u_{\ell,m}
    =
    \partial_{rr}u_{\ell,m}
    +
    \frac{N-1}{r}\partial_r u_{\ell,m}
    +
    \frac{\lambda-\ell(\ell+N-2)}{r^2}u_{\ell,m}.
\]
Thus the \((\ell,m)\)-component satisfies a radial Hardy heat equation with 
parameter $\lambda-\ell(\ell+N-2). $
Since
\[
    \mu^2-\lambda_\ell
    =
    \mu^2+\ell(\ell+N-2)-\lambda
    =
    (\ell+\mu)^2-\lambda,
\]
applying Proposition \ref{radialkernel} with parameter $\nu_\ell=\sqrt{(\ell+\mu)^2-\lambda}.$, we have
\[
    u_{\ell,m}(r,t)
    =
    \int_0^\infty
    K_{\ell}(r,\rho,t)
    u_{0,\ell,m}(\rho)\rho^{N-1}\,d\rho,
\]
where
\[
    K_{\ell}(r,\rho,t)
    =
    \frac{1}{\left|\mathbb{S}^{N-1}\right|}\frac{1}{2t}
    (r\rho)^{-\mu}
    \exp\!\left(-\frac{r^2+\rho^2}{4t}\right)
    I_{\nu_\ell}\left(\frac{r\rho}{2t}\right).
\]
Therefore
\[
\begin{aligned}
    u(r,\omega,t)
    &=
    \sum_{\ell=0}^{\infty}
    \sum_{m=1}^{d_\ell}
    Y_{\ell,m}(\omega)
    \int_0^\infty
    K_{\ell}(r,\rho,t)
    u_{0,\ell,m}(\rho)
    \rho^{N-1}\,d\rho  \\
    &=
    \int_0^\infty
    \int_{\mathbb S^{N-1}}
    \left[
    \sum_{\ell=0}^{\infty}
    \sum_{m=1}^{d_\ell}
    K_{\ell}(r,\rho,t)
    Y_{\ell,m}(\omega)Y_{\ell,m}(\eta)
    \right]
    u_0(\rho,\eta)
    \rho^{N-1}\,d\sigma(\eta)\,d\rho.
\end{aligned}
\]
Since \(u_0\in C_c^\infty(\mathbb R^N\setminus\{0\})\), the radial variable \(\rho\) stays in a compact subset of \((0,\infty)\). Moreover, the Bessel factor decays rapidly in \(\ell\), while the spherical harmonics grow at most polynomially. Hence the series is absolutely and locally uniformly convergent, and the interchange of the series with the integrals is justified.

Since $dy=\rho^{N-1}\,d\rho\,d\sigma(\eta),$ this gives the representation
\[
    u(x,t)
    =
    \int_{\mathbb R^N}
    K(x,y,t)u_0(y)\,dy,
\]
with
\[
K(x,y,t)
=
\sum_{\ell=0}^{\infty}
\sum_{m=1}^{d_\ell}
K_{\ell}(r,\rho,t)
Y_{\ell,m}(\omega)Y_{\ell,m}(\eta).
\]
Substituting the expression of \(K_{\ell}\) gives the desired formula.

\end{proof}

\begin{remark} The preceding formula also defines the solution for less regular radial data using the standard density argument. More precisely, assume that \(u_0\) is radial and $|x|^{\nu-\mu}u_0\in L^1(\mathbb R^N). $
Then, for every \(t>0\) and \(x\neq0\), the integral
\[
u(x,t):=\int_{\mathbb R^N}K(x,y,t)u_0(y)\,dy
\]
is well defined.

Indeed, by the small-argument asymptotic
\[
I_\nu(z)\sim \frac{z^\nu}{2^\nu\Gamma(\nu+1)},
\qquad z\to0,
\]
we have, for fixed \(x\neq0\) and \(t>0\),
\[
K(x,y,t)\le C_{x,t}|y|^{\nu-\mu}
\qquad \text{as } |y|\to0.
\]
On the other hand, for large \(|y|\), the Gaussian factor gives exponential decay. Hence the above assumptions guarantee absolute convergence of the kernel integral.

Moreover, by differentiating under the integral sign on compact subsets of
\[
(\mathbb R^N\setminus\{0\})\times(0,\infty),
\]
the function \(u\) belongs to
\[
C^\infty\bigl((\mathbb R^N\setminus\{0\})\times(0,\infty)\bigr)
\]
and satisfies
\[
u_t=\Delta u+\lambda |x|^{-2}u
\]
pointwise for \(t>0\) and \(x\neq0\). Thus \(u\) is a classical solution away from $x=0$.
\end{remark}

\begin{lemma}\label{gammaint}
Let \(N\ge1\), \(1\le p<\infty\), \(a\in\mathbb R\), and \(c>0\). Assume
\[
ap+N>0.
\]
Then, for every \(t>0\),
\[
\left\|
|x|^a e^{-c\frac{|x|^2}{t}}
\right\|_{L^p(\mathbb R^N)}
=
C_{N,p,a,c}\,
t^{\frac a2+\frac{N}{2p}},
\]
where
\[
C_{N,p,a,c}
=
\left[
\frac{|\mathbb S^{N-1}|}{2}
(cp)^{-\frac{ap+N}{2}}
\Gamma\left(\frac{ap+N}{2}\right)
\right]^{1/p}.
\]
\end{lemma}

\begin{proof}
By definition,
\[
\left\|
|x|^a e^{-c\frac{|x|^2}{t}}
\right\|_{L^p(\mathbb R^N)}^p
=
\int_{\mathbb R^N}
|x|^{ap}
e^{-cp\frac{|x|^2}{t}}
\,dx.
\]
Using polar coordinates, we obtain
\[
\int_{\mathbb R^N}
|x|^{ap}
e^{-cp\frac{|x|^2}{t}}
\,dx
=
|\mathbb S^{N-1}|
\int_0^\infty
r^{ap+N-1}
e^{-cp\frac{r^2}{t}}
\,dr.
\]
Set
\[
r=\left(\frac{t}{cp}\right)^{1/2}s^{1/2},
\qquad
dr=
\frac12
\left(\frac{t}{cp}\right)^{1/2}
s^{-1/2}\,ds.
\]
Therefore
\[
\begin{aligned}
\int_{\mathbb R^N}
|x|^{ap}
e^{-cp\frac{|x|^2}{t}}
\,dx
&=
\frac{|\mathbb S^{N-1}|}{2}
\left(\frac{t}{cp}\right)^{\frac{ap+N}{2}}
\int_0^\infty
s^{\frac{ap+N}{2}-1}e^{-s}\,ds
\\
&=
\frac{|\mathbb S^{N-1}|}{2}
(cp)^{-\frac{ap+N}{2}}
\Gamma\left(\frac{ap+N}{2}\right)
t^{\frac{ap+N}{2}}.
\end{aligned}
\]
The Gamma integral is finite exactly when
\[
\frac{ap+N}{2}>0\iff ap+N>0,
\]
Thus
\[
\left\|
|x|^a e^{-c\frac{|x|^2}{t}}
\right\|_{L^p(\mathbb R^N)}^p
=
\frac{|\mathbb S^{N-1}|}{2}
(cp)^{-\frac{ap+N}{2}}
\Gamma\left(\frac{ap+N}{2}\right)
t^{\frac{ap+N}{2}}.
\]
Taking the \(p\)-th root gives
\[
\left\|
|x|^a e^{-c\frac{|x|^2}{t}}
\right\|_{L^p(\mathbb R^N)}
=
\left[
\frac{|\mathbb S^{N-1}|}{2}
(cp)^{-\frac{ap+N}{2}}
\Gamma\left(\frac{ap+N}{2}\right)
\right]^{1/p}
t^{\frac a2+\frac{N}{2p}}.
\]
This proves the claim.
\end{proof}
\begin{lemma}\label{localization}
Let \(c>0\), \(c_0>0\), \(1\le p<\infty\), and \(\beta\in\mathbb R\). Define
\[
H(\eta)
:=
\left(
\int_{c_0/\eta}^{\infty}
s^\beta e^{-c(s-\eta)^2}\,ds
\right)^{1/p},
\qquad \eta>0.
\]
Then the following estimates hold.

\begin{enumerate}
    \item For every \(M>0\), there exists \(C_M>0\) such that
    \[
    H(\eta)\le C_M \eta^M,
    \qquad 0<\eta\le 1.
    \]

    \item There exists \(C>0\) such that
    \[
    H(\eta)\le C\eta^{\beta/p},
    \qquad \eta\ge 1.
    \]
\end{enumerate}
\end{lemma}
\begin{proof}
Let
\[
\eta_0:=\min\left\{1,\sqrt{\frac{c_0}{2}}\right\}\leq 1\leq
 \eta_1:=\max\left\{1,\sqrt{2c_0}\right\}.
\]

If \(0<\eta\le \eta_0\), then $
s\ge \frac{c_0}{\eta}$ implies $s/2\ge \eta.$ Hence $
e^{-c(s-\eta)^2}
\le
e^{-cs^2/4}.$ Therefore
\[
\int_{c_0/\eta}^{\infty}
s^\beta e^{-c(s-\eta)^2}\,ds
\le
\int_{c_0/\eta}^{\infty}
s^\beta e^{-cs^2/4}\,ds.
\]
The Gaussian term decays faster than any power. Hence, for every \(L>0\),
there exists \(C_L>0\) such that
\[
\int_{c_0/\eta}^{\infty}
s^\beta e^{-cs^2/4}\,ds
\le
C_L \eta^L,
\qquad 0<\eta\le \eta_0.
\]
Choosing \(L=Mp\), we get
\[
H(\eta)\le C_M\eta^M,
\qquad 0<\eta\le \eta_0.
\]
On the compact interval \([\eta_0,1]\), the function \(H\) is uniformly
bounded. Indeed, if \(\eta\in[\eta_0,1]\), then $\frac{c_0}{\eta}\ge c_0, $
so the lower limit of integration stays bounded away from zero. Hence the
factor \(s^\beta\) has no singularity on the interval of integration, even if
\(\beta<0\). Moreover, the exponential term gives Gaussian decay as
\(s\to\infty\). Therefore there exists \(C_0>0\) such that
\[
H(\eta)\le C_0,
\qquad \eta\in[\eta_0,1].
\]
Since \(\eta\ge\eta_0\) on \([\eta_0,1]\), we have $
\eta^M\ge \eta_0^M. $
Equivalently, $
1\le \eta_0^{-M}\eta^M.$ Thus
\[
H(\eta)\le C_0
\le C_0\eta_0^{-M}\eta^M,
\qquad \eta\in[\eta_0,1].
\]
Therefore, after substituting $C_M:=C_0\eta_0^{-M}$, we obtain
\[
H(\eta)\le C_M\eta^M,
\qquad \eta_0\le\eta\le1.
\]
Similarly, since the function \(\eta\mapsto \eta^{\beta/p}\) is positive and
continuous on \([1,\eta_1]\), there exists \(m>0\) such that
\[
\eta^{\beta/p}\ge m,
\qquad \eta\in[1,\eta_1].
\]
Therefore
\[
H(\eta)\le C_1
\le C_1m^{-1}\eta^{\beta/p}:=C\eta^{\beta/p},
\qquad \eta\in[1,\eta_1].
\]
We now prove the estimate for \(\eta\ge \eta_1\). Since $\eta/2\geq c_0/\eta$,  we split the integral as
\[
\int_{c_0/\eta}^{\infty}
s^\beta e^{-c(s-\eta)^2}\,ds
=
I_1+I_2+I_3,
\]
where
\[
I_1=
\int_{c_0/\eta}^{\eta/2}
s^\beta e^{-c(s-\eta)^2}\,ds,
\]
\[
I_2=
\int_{\eta/2}^{3\eta/2}
s^\beta e^{-c(s-\eta)^2}\,ds,
\]
and
\[
I_3=
\int_{3\eta/2}^{\infty}
s^\beta e^{-c(s-\eta)^2}\,ds.
\]

On the middle region, $\displaystyle\frac{\eta}{2}\leq s\leq \frac{3\eta}{2}$. Hence $
s^\beta\le C\eta^\beta, $
and therefore
\[
I_2
\le
C\eta^\beta
\int_{\eta/2}^{3\eta/2}
e^{-c(s-\eta)^2}\,ds
\le C\eta^{\beta}\int_{-\infty}^{\infty}e^{-cx^2}dx\leq
C\eta^\beta.
\]

On the left region, \(s\le\eta/2\), so $
|s-\eta|\ge \frac{\eta}{2}. $
Thus
\[
e^{-c(s-\eta)^2}
\le
e^{-c\eta^2/4}.
\]
Therefore
\[
I_1
\le
e^{-c\eta^2/4}
\int_{c_0/\eta}^{\eta/2}s^\beta\,ds.
\]
The last integral grows at most like a fixed power of \(\eta\). More precisely,
there exists \(K>0\), depending only on \(\beta\), such that
\[
\int_{c_0/\eta}^{\eta/2}s^\beta\,ds
\le
C\eta^K,
\qquad \eta\ge1.
\]
Since the exponential factor dominates every power of \(\eta\), we get
\[
I_1\le C\eta^\beta,
\qquad \eta\ge1.
\]

It remains to estimate \(I_3\). Put
\[
\sigma=s-\eta.
\]
Then
\[
I_3
=
\int_{\eta/2}^{\infty}
(\eta+\sigma)^\beta e^{-c\sigma^2}\,d\sigma.
\]
If \(\beta\le0\), then for \(s=\eta+\sigma\ge3\eta/2\), thus $(\eta+\sigma)^\beta\le C\eta^\beta, $
and hence
\[
I_3\le C\eta^\beta
\int_{\eta/2}^{\infty}e^{-c\sigma^2}\,d\sigma
\le
C\eta^\beta.
\]
If \(\beta>0\), then, since \(\eta\ge1\),
\[
(\eta+\sigma)^\beta
\le
C\eta^\beta(1+\sigma)^\beta.
\]
Thus
\[
I_3
\le
C\eta^\beta
\int_{\eta/2}^{\infty}
(1+\sigma)^\beta e^{-c\sigma^2}\,d\sigma
\le
C\eta^\beta.
\]
Combining the estimates for \(I_1,I_2,I_3\), we obtain
\[
\int_{c_0/\eta}^{\infty}
s^\beta e^{-c(s-\eta)^2}\,ds
\le
C\eta^\beta,
\qquad \eta\ge1.
\]
Taking the \(p\)-th root gives $
H(\eta)\le C\eta^{\beta/p},
\qquad \eta\ge1. $
completing the proof.
\end{proof}
\begin{lemma}\label{bumpfunctions}
Let $n\geq 1$. Then there exist radial functions
\[
\eta_0,\eta_1,\dots,\eta_n\in C_c^\infty(\mathbb R^N)
\]
such that
\[
M_{\nu,i}(\eta_j)=\delta_{ij},
\qquad 0\le i,j\le n.
\]
\end{lemma}

\begin{proof}
Choose \(n+1\) distinct radii
\[
0<r_0<r_1<\cdots<r_n.
\]
For each \(k=0,\dots,n\), choose a nonnegative radial bump function
\[
\rho_{k,\varepsilon}\in C_c^\infty(\mathbb R^N)
\]
supported in the annulus
\[
\{x\in\mathbb R^N: ||x|-r_k|<\varepsilon\},
\]
where \(\varepsilon>0\) is small enough so that these annuli are pairwise
disjoint. Thus the moments $M_{\nu,j}(\rho_{k,\epsilon})$ are finite, and, without loss of generality we can assume
\[
M_{\nu,0}(\rho_{k,\varepsilon})
=
\int_{\mathbb R^N}|x|^{\nu-\mu}\rho_{k,\varepsilon}(x)\,dx
=1.
\]
For \(i=0,\dots,n\), we then have
\begin{equation}
    \begin{split}
        |M_{\nu,i}(\rho_{k,\varepsilon})-r_k^{2i}|
&=
\left|\int_{\mathbb R^N}|x|^{\nu-\mu+2i}\rho_{k,\varepsilon}(x)\,dx -r_k^{2i}\right|\\
&\leq\int_{\mathbb R^N}\left||x|^{2i}-r_k^{2i}\right||x|^{\nu-\mu}\rho_{k,\varepsilon}(x)\,dx.\\
&\leq \sup_{||x|-r_k|<\varepsilon}\left||x|^{2i}-r_k^{2i}\right|\int_{\R^N}|x|^{\nu-\mu}\rho_{k,\varepsilon}(x)dx\\
&=\sup_{||x|-r_k|<\varepsilon}\left||x|^{2i}-r_k^{2i}\right|\stackrel{\varepsilon\to 0}{\longrightarrow} 0
    \end{split}
\end{equation}
Hence $M_{\nu,i}(\rho_{k,\varepsilon})\to r_k^{2i}$ as $\varepsilon\to 0$. Consider the matrix
\[
B_\varepsilon=(B_{ik})_{0\le i,k\le n},
\qquad
B_{ik}:=M_{\nu,i}(\rho_{k,\varepsilon}).
\]
By the previous convergence, $b_\varepsilon\longrightarrow V
\qquad \text{as } \varepsilon\to0, $ where $V=(r_k^{2i})_{0\le i,k\le n}. $
The matrix \(V\) is a Vandermonde matrix associated with the distinct numbers
\[
r_0^2,r_1^2,\dots,r_n^2,
\]
hence \(V\) is invertible. Therefore, for \(\varepsilon>0\) sufficiently small,
\(B_\varepsilon\) is also invertible.

Fix such an \(\varepsilon\), and write \(\rho_k=\rho_{k,\varepsilon}\) and
\(B=B_\varepsilon\). Define
\[
\eta_j
:=
\sum_{k=0}^{n}
(B^{-1})_{kj}\rho_k,
\qquad j=0,\dots,n.
\]
Then, for every \(0\le i,j\le n\),
\[
\begin{aligned}
M_{\nu,i}(\eta_j)
&=
\sum_{k=0}^{n}
(B^{-1})_{kj}M_{\nu,i}(\rho_k)  \\
&=
\sum_{k=0}^{n}
B_{ik}(B^{-1})_{kj} \\
&=
\delta_{ij}.
\end{aligned}
\]
This proves the claim.
\end{proof}
\section*{Acknowledgements}

I express my gratitude to my Ph.D. supervisor, Prof. Dr. Cristian Cazacu, for his
guidance, support, and valuable comments throughout the preparation of this work.
This work was partially supported by a doctoral fellowship from the Doctoral
School of Mathematics, University of Bucharest.

\end{document}